\documentclass[a4paper,oneside,11pt]{article}
\usepackage[a4paper]{geometry}
\usepackage{aeguill}                 
\usepackage{graphicx}
\usepackage{amsmath,amsfonts,amssymb,amsthm}
\usepackage{pstricks} 
\usepackage{epstopdf}
\usepackage{srcltx}
\usepackage[utf8]{inputenc} 
\usepackage{hyperref} 
\usepackage[english]{babel} 
\usepackage{comment}

\title{Orbit equivalence rigidity of irreducible actions of right-angled Artin groups}

\author{Camille Horbez, Jingyin Huang and Adrian Ioana}

\usepackage[all]{xy}
\usepackage{bbm}

\begin{document}

\maketitle 

\newtheorem{de}{Definition} [section]
\newtheorem{theo}[de]{Theorem} 
\newtheorem{prop}[de]{Proposition}
\newtheorem{lemma}[de]{Lemma}
\newtheorem{cor}[de]{Corollary}
\newtheorem{propd}[de]{Proposition-Definition}
\newtheorem{conj}[de]{Conjecture}
\newtheorem{claim}{Claim}
\newtheorem*{claim2}{Claim}
\newtheorem{theointro}{Theorem}
\newtheorem*{defintro}{Definition}
\newtheorem{corintro}[theointro]{Corollary}
\newtheorem{questionintro}[theointro]{Question}
\newtheorem{propintro}[theointro]{Proposition}

\theoremstyle{remark}
\newtheorem{rk}[de]{Remark}
\newtheorem{ex}[de]{Example}
\newtheorem{question}[de]{Question}

\normalsize

\newcommand{\Aut}{\mathrm{Aut}}
\newcommand{\Out}{\mathrm{Out}}
\newcommand{\Inn}{\mathrm{Inn}}
\newcommand{\stab}{\operatorname{Stab}}
\newcommand{\pstab}{\operatorname{Pstab}}
\newcommand{\dunion}{\sqcup}
\newcommand{\eps}{\varepsilon}
\renewcommand{\epsilon}{\varepsilon}
\newcommand{\calf}{\mathcal{F}}
\newcommand{\cali}{\mathcal{I}}
\newcommand{\caly}{\mathcal{Y}}
\newcommand{\calx}{\mathcal{X}}
\newcommand{\calz}{\mathcal{Z}}
\newcommand{\calo}{\mathcal{O}}
\newcommand{\calb}{\mathcal{B}}
\newcommand{\calq}{\mathcal{Q}}
\newcommand{\calu}{\mathcal{U}}
\newcommand{\call}{\mathcal{L}}
\newcommand{\bbR}{\mathbb{R}}
\newcommand{\bbZ}{\mathbb{Z}}
\newcommand{\bbD}{\mathbb{D}}
\newcommand{\NT}{\mathrm{NT}}
\newcommand{\cat}{\mathrm{CAT}(-1)}
\newcommand{\CAT}{\mathrm{CAT}(0)}
\newcommand{\actson}{\curvearrowright}
\newcommand{\caln}{\mathcal{N}}
\newcommand{\calg}{\mathcal{G}}
\newcommand{\Prob}{\mathrm{Prob}}
\newcommand{\calt}{\mathcal{T}}
\newcommand{\calc}{\mathcal{C}}
\newcommand{\adm}{\mathrm{adm}}
\newcommand{\cala}{\mathcal{A}}
\newcommand{\cals}{\mathcal{S}}
\newcommand{\calh}{\mathcal{H}}
\newcommand{\Stab}{\mathrm{Stab}}
\newcommand{\bdd}{\mathrm{bdd}}
\newcommand{\calp}{\mathcal{P}}
\newcommand{\Fix}{\mathrm{Fix}}
\newcommand{\cald}{\mathcal{D}}
\newcommand{\rad}{r}
\newcommand{\good}{\mathrm{good}}
\newcommand{\Faces}{\mathrm{Faces}}
\newcommand{\del}{\mathbb{D}_\Gamma}
\newcommand{\Mod}{\mathrm{Mod}}
\newcommand{\Comm}{\mathrm{Comm}}
\newcommand{\si}{\sigma}
\newcommand{\st}{\mathrm{st}}
\newcommand{\lk}{\mathrm{lk}}
\newcommand{\supp}{\mathrm{supp}}
\newcommand{\HHS}{\mathrm{HHS}}
\newcommand{\Creg}{\mathcal{C}_{\mathrm{reg}}}
\newcommand{\reg}{\mathrm{reg}}
\newcommand{\calk}{\mathcal{K}}
\newcommand{\calr}{\mathcal{R}}
\newcommand{\cale}{\mathcal{E}}
\newcommand{\bddc}{\mathrm{H}^2_{\mathrm{b}}}
\newcommand{\Ufin}{\mathcal{U}_{\mathrm{fin}}}
\newcommand{\Ginv}{\mathcal{G}_{\mathrm{inv}}}

\makeatletter

\edef\@tempa#1#2{\def#1{\mathaccent\string"\noexpand\accentclass@#2 }}

\@tempa\rond{017}

\makeatother

\begin{abstract}
Let $G_\Gamma\actson X$ and $G_\Lambda\actson Y$ be two free measure-preserving actions of one-ended right-angled Artin groups with trivial center on standard probability spaces. Assume they are \emph{irreducible}, i.e.\ every element from a standard generating set acts ergodically. We prove that if the two actions are stably orbit equivalent (or merely stably $W^*$-equivalent), then they are automatically conjugate through a group isomorphism between $G_\Gamma$ and $G_\Lambda$. Through work of Monod and Shalom, we derive a superrigidity statement: if the action $G_\Gamma\actson X$ is stably orbit equivalent (or merely stably $W^*$-equivalent) to a free, measure-preserving, mildly mixing action of a countable group, then the two actions are virtually conjugate.  We also use works of Popa and Ioana-Popa-Vaes to establish the $W^*$-superrigidity of Bernoulli actions of all ICC groups having a finite generating set made of infinite-order elements where two consecutive elements commute, and one has a nonamenable centralizer: these include one-ended non-abelian right-angled Artin groups, but also many other Artin groups and most mapping class groups of finite-type surfaces. 
\end{abstract}

\section*{Introduction}

A main goal in measured group theory, initiated by work of Dye \cite{Dye}, is to classify measure-preserving group actions on standard probability spaces up to \emph{orbit equivalence}, i.e.\ up to the existence of a measure space isomorphism sending orbits to orbits. More generally, we will be interested in \emph{stable orbit equivalence} of actions of countable groups, defined as follows: two free, ergodic, measure-preserving actions $G\actson X$ and $H\actson Y$ by Borel automorphisms on standard probability spaces are \emph{stably orbit equivalent} if there exist positive measure Borel subsets $U\subseteq X$ and $V\subseteq Y$, and a measure-scaling isomorphism $f:U\to V$, such that for every $x\in U$, one has $f((G\cdot x)\cap U)=(H\cdot f(x))\cap V$. 

A first striking result in this theory was the proof by Ornstein and Weiss \cite{OW}, building on Dye's work, that any two free, ergodic, probability measure-preserving actions of countably infinite amenable groups are orbit equivalent. 

Later, Gaboriau used the notion of cost (introduced by Levitt in \cite{Lev}) to distinguish actions of free groups of different ranks \cite{Gab-cost}, and showed that $\ell^2$-Betti numbers also provide useful invariants for the classification \cite{Gab-l2}.

In contrast to the Ornstein--Weiss theorem exhibiting a wide class of groups that are indistinguishable from the viewpoint of orbit equivalence, several strong rigidity results have then been obtained for various classes of groups, like higher rank lattices (Furman \cite{Fur-me,Fur-oe}), mapping class groups (Kida \cite{Kid,Kid-oe}) and related groups (e.g.\ \cite{CK}), certain large type Artin groups \cite{HH1} or $\mathrm{Out}(F_N)$ with $N\ge 3$ (as proved by Guirardel and the first named author in \cite{GH}). Interestingly, negative curvature features of the groups under consideration are often key ingredients in the proofs of orbit equivalence rigidity of their ergodic actions. 

Other rigidity phenomena were discovered by Monod and Shalom \cite{MS}, who proved superrigidity-type results for \emph{irreducible} actions of direct products of free groups, or more generally of direct products $G_1\times\dots\times G_k$, with $k\ge 2$, where $\bddc(G_i,\ell^2(G_i))\neq 0$ for every $i\in\{1,\dots,k\}$ (this condition on the bounded cohomology can be viewed as an analytical form of negative curvature). The crucial \emph{irreducibility} assumption means that every factor $G_i$ acts ergodically on $X$. 

In yet another direction, Popa obtained orbit equivalence rigidity results for Bernoulli actions of all property (T) groups \cite{Pop-T}, and all nonamenable groups that split as direct products or have an infinite center \cite{Pop}; these results were obtained in the framework of Popa's deformation/rigidity theory, and their proofs exploit a specific property of Bernoulli actions called \emph{malleability}, rather than geometric properties of the acting group.

In \cite{HH2}, we started to investigate the class of right-angled Artin groups from the viewpoint of measured group theory. These groups are of basic importance (see e.g. \cite{charney2007introduction,Wise}) and have a very simple definition: given a finite simple graph $\Gamma$ (i.e.\ with no loop-edge and no multiple edges between two vertices), the \emph{right-angled Artin group} $G_\Gamma$ is defined by the following presentation: it has one generator per vertex of $\Gamma$, and relations are given by commutation of any two generators whose associated vertices are joined by an edge. 

On the rigidity side, we proved in \cite{HH2} that if two right-angled Artin groups $G_\Gamma,G_\Lambda$ with finite outer automorphism groups admit free, ergodic, measure-preserving actions on standard probability spaces which are orbit equivalent, or merely stably orbit equivalent (equivalently if the groups are measure equivalent), then  $G_\Gamma$ and $G_\Lambda$ are isomorphic. However, rigidity fails beyond this context: given any right-angled Artin group $G_\Gamma$, and any group $H$ which is a graph product of countably infinite amenable groups over the same graph $\Gamma$, we can build free, ergodic, probability measure-preserving actions of $G_\Gamma$ and $H$ which are orbit equivalent \cite[Proposition~4.2]{HH2}. In fact, our proof of \cite[Proposition~4.2]{HH2} shows that starting from any action $G_\Gamma\actson Z$ as above, we can find a blown-up action $G_\Gamma\actson\hat{Z}$ (i.e.\ coming with a $G_\Gamma$-equivariant map $\hat{Z}\to Z$) which fails to be superrigid for orbit equivalence. We can also build two actions of $G_\Gamma$ which are orbit equivalent but not conjugate \cite[Remark~4.4]{HH2}.    

The goal of the present paper is to show that rigidity can be achieved if one restricts to a certain class of actions satisfying more restrictive ergodicity conditions, as in the following definition.  

\begin{defintro}
Let $G$ be a right-angled Artin group. A free, probability measure-preserving action of $G$ on a standard probability space $X$ is \emph{irreducible} if there exist a finite simple graph $\Gamma$ and an isomorphism between $G$ and the right-angled Artin group $G_\Gamma$ such that, through this isomorphism, every standard generator of $G_\Gamma$ (associated to a vertex of $\Gamma$) acts ergodically on $X$. 
\end{defintro}

The above definition is a natural extension of Monod and Shalom's irreducibility condition to the context of right-angled Artin groups (and could be naturally extended to graph products). Examples of irreducible actions of right-angled Artin groups include Bernoulli actions (considered in Theorem~\ref{theointro:Bernoulli} below) and Gaussian actions associated to mixing orthogonal representations (introduced by Connes and Weiss in \cite{CW}, see also \cite[Section~2.1]{PS} for a detailed study). To build other examples, one can start with a discrete and faithful representation of a right-angled Artin group into $\mathrm{SL}(n,\mathbb{R})$ or even $\mathrm{SL}(n,\mathbb{Z})$ (see \cite{Wan} for examples). Using this embedding in $\mathrm{SL}(n,\mathbb{R})$, one can then consider the restriction of a mixing action of a closed subgroup of $\mathrm{SL}(n,\mathbb{R})$ on a homogeneous space, coming from the Howe--Moore theorem \cite{HM}, see e.g.\ \cite[Corollary~2.5]{Bek}. Our main theorem is the following.

\begin{theointro}\label{theointro:strong-rigidity}
Let $G$ and $H$ be two one-ended right-angled Artin groups with trivial center. Let $G\actson X$ and $H\actson Y$ be two free irreducible measure-preserving actions by Borel automorphisms on standard probability spaces. 

If the actions $G\actson X$ and $H\actson Y$ are stably orbit equivalent (or merely stably $W^*$-equivalent\footnote{i.e.\ their associated von Neumann algebras $L^\infty(X)\rtimes G$ and $L^\infty(Y)\rtimes H$, defined via Murray and von Neumann's \emph{group measure space construction} \cite{MvN}, have isomorphic amplifications}), then they are conjugate, i.e.\ there exist a group isomorphism $\alpha:G\to H$ and a measure space isomorphism $f:X\to Y$ such that for every $g\in G$ and almost every $x\in X$, one has $f(gx)=\alpha(g)f(x)$.
\end{theointro}

Theorem~\ref{theointro:strong-rigidity} covers a much larger class of right-angled Artin groups than our previous work \cite{HH2}, including many examples with infinite outer automorphism group. 

For example, it applies to all right-angled Artin groups whose defining graph is a tree of diameter at least $3$, which are usually less rigid from other viewpoints (for instance, they are all quasi-isometric \cite{BN}, and the problem of their measure equivalence classification is open). Also, contrary to our previous work (and to other measure equivalence rigidity statements in the literature, like \cite{Kid,HH1,GH}), our proof of Theorem~\ref{theointro:strong-rigidity} does not rely on a combinatorial rigidity statement for a curve graph analogue \cite{KK} associated to the right-angled Artin group. Instead, rigidity comes from the combination of a local argument (untwisting the orbit equivalence cocycle to a group homomorphism inside a vertex group), and a propagation argument where the commutation relations play a central role. The irreducibility assumption is crucial in both steps. The first step relies on a new orbit equivalence invariant of right-angled Artin groups (compared to \cite{HH2}), namely, the orbit equivalence relation remembers the maximal join subgroups of $G$ and $H$; this is important as it enables us to apply the results of Monod and Shalom in these local subgroups as a crucial step of the proof. 

As already explained above, counterexamples without the irreducibility assumption were given in \cite[Section~4.1]{HH2}. Counterexamples when the groups are infinitely-ended already arise in the context of free groups. Indeed, Bowen proved in \cite{Bow} that all nontrivial Bernoulli shifts of a given finitely generated free group are orbit equivalent; more generally, if $G=A_1\ast\dots\ast A_n$ and $G'=A'_1\ast\dots\ast A'_n$ are two free products of countably infinite amenable groups with the same number of factors, then all Bernoulli shifts of $G$ and $G'$ are orbit equivalent. In these contexts, the Bernoulli shifts are completely classified up to conjugation by the entropy of their base space \cite{Bow3,Bow4}, yielding a $1$-parameter family of orbit equivalent pairwise nonconjugate actions. He also proved that all nontrivial Bernoulli shifts of finitely generated nonabelian free groups (possibly of different ranks) are stably orbit equivalent \cite{Bow2} -- although as already mentioned, work of Gaboriau ensures that they are not orbit equivalent when the ranks of the acting groups are different, by comparing their costs \cite{Gab-cost}. This is in sharp contrast with our Theorem~\ref{theointro:strong-rigidity}, where stably orbit equivalent irreducible actions are automatically orbit equivalent, and in fact even conjugate. 

We mention that in the context of right-angled Artin groups, the stable $W^*$-rigidity statement in Theorem~\ref{theointro:strong-rigidity} is a consequence of the stable orbit equivalence rigidity statement, using that the corresponding von Neumann algebras have a unique virtual Cartan subalgebra up to unitary conjugacy. Uniqueness of the virtual Cartan subalgebra up to unitary conjugacy was proved in a groundbreaking work of Popa and Vaes \cite[Theorem~1.2 and Remark~1.3]{PV2} for all free, ergodic, probability measure-preserving actions of groups satisfying Ozawa and Popa's property $(\mathrm{HH})^+$  -- the fact that right-angled Artin groups satisfy this property was established by Ozawa and Popa in \cite[Theorem~2.3(5)]{OP}. See also \cite[Corollary 3.20]{HH2} for a more detailed explanation, and recent work of Chifan and Kunnawalkam Elayavalli for the more general case of graph products \cite{CKE}. 

We also mention that we actually obtain a slightly stronger statement than Theorem~\ref{theointro:strong-rigidity}, namely: every stable orbit equivalence between the actions $G\actson X$ and $H\actson Y$ has compression 1 (see Section~\ref{sec:soe} for definitions, and Propositions~\ref{theo:join-case} and~\ref{theo:coned-case} for our precise statements). In particular, the \emph{fundamental group} of the equivalence relation $\mathcal{R}$ associated to the action $G\actson X$ (i.e.\ the subgroup of $\mathbb{R}_+^*$ consisting of all $t>0$ such that $\mathcal{R}$ is isomorphic to the amplification $\mathcal{R}^t$) is trivial. Notice that the class of one-ended right-angled Artin groups with trivial center contains groups whose $\ell^2$-Betti numbers all vanish (e.g.\ all right-angled Artin groups whose defining graph is a tree of diameter at least $3$, see \cite{DL}), and for these triviality of the fundamental group does not follow from Gaboriau's proportionality principle \cite{Gab-l2}. As a consequence, the fundamental group of the group measure space von Neumann algebra $L^\infty(X)\rtimes G$ (defined by Murray and von Neumann in \cite{MvN,MvN2} as the subgroup of $\mathbb{R}_+^*$ consisting of all $t>0$ such that $L^\infty(X)\rtimes G$ is isomorphic to the amplification $(L^\infty(X)\rtimes G)^t$) is also trivial. Indeed, this again follows from the analogous result for $\calr$ precisely because $L^\infty(X)\rtimes G$ has a unique  Cartan subalgebra -- this reduction is at the heart of many remarkable results in deformation/rigidity theory \cite{Pop3}. 

\medskip

Using general techniques from measured group theory, developed in successive works of Furman \cite{Fur-oe}, Monod and Shalom \cite{MS}, and Kida \cite{Kid-oe}, Theorem~\ref{theointro:strong-rigidity} yields a superrigidity theorem within the class of mildly mixing group actions. Recall that an action of a countable group $G$ on a standard probability space $X$ is \emph{mildly mixing} if for every non-singular properly ergodic action of $G$ on a standard probability measure space $Y$, the diagonal $G$-action on $X\times Y$ is ergodic. Recall also that two measure-preserving actions $G_1\actson X_1$ and $G_2\actson X_2$ of countable groups on standard probability spaces are \emph{virtually conjugate} if there exist short exact sequences $1\to F_i\to G_i\to \bar{G}_i\to 1$ with $F_i$ finite, finite-index subgroups $\bar{G}_i^0\subseteq\bar{G}_i$, and conjugate actions $\bar{G}^0_i\actson X'_i$ (through an isomorphism between $\bar{G}_1^0$ and $\bar{G}_2^0$) such that for every $i\in\{1,2\}$, the action $\bar{G}_i\actson X_i/F_i$ is induced from $\bar{G}_i^0\actson X'_i$ as in \cite[Definition~2.1]{Kid-oe}.  

\begin{theointro}\label{theointro:superrigidity}
Let $G$ be a one-ended right-angled Artin group with trivial center. Let $G\actson X$ be a free, irreducible, measure-preserving action of $G$ on a standard probability space $X$. Let $H$ be a countable group, and let $H\actson Y$ be a mildly mixing, free, measure-preserving action of $H$ on a standard probability space $Y$.

If the actions $G\actson X$ and $H\actson Y$ are stably orbit equivalent (or merely stably $W^*$-equivalent), then they are virtually conjugate.  
\end{theointro}

In the specific case of nontrivial \emph{Bernoulli actions} of $G$ (i.e.\ of the form $G\actson X_0^G$, where $X_0$ is a standard probability space not reduced to a single atom, and the action is by shift), an even stronger conclusion holds, which does not require any mildly mixing assumption on the $H$-action. By exploiting works of Popa \cite{Pop} and of Ioana, Popa and Vaes \cite{IPV}, we reach the following statement.

\begin{theointro}\label{theointro:Bernoulli}
Let $G$ be an ICC countable group, which admits a finite generating set $S=\{s_1,\dots,s_k\}$ made of infinite-order elements, such that for every $i\in\{1,\dots,k-1\}$, the elements $s_i$ and $s_{i+1}$ commute, and $s_1$ has a nonamenable centralizer in $G$. 

Let $G\actson X$ be a nontrivial Bernoulli action of $G$. Let $H$ be a countable group, and let $H\actson Y$ be a free, ergodic, measure-preserving action of $H$ on a standard probability space $Y$.

If the actions $G\actson X$ and $H\actson Y$ are orbit equivalent (or merely $W^*$-equivalent), then they are conjugate.
\end{theointro}

This applies to all one-ended nonabelian right-angled Artin groups: in fact in this case, using the uniqueness of the virtual Cartan subalgebra up to unitary conjugacy, we also obtain that if the actions $G\actson X$ and $H\actson Y$ are stably $W^*$-equivalent, then they are virtually conjugate. The above theorem also applies to many (non-right-angled) Artin groups and to most mapping class groups of finite-type orientable surfaces. Let us also mention that the $W^*$-superrigidity of Bernoulli actions of countable ICC Property~(T) groups was proved by Ioana in \cite{Ioa-T}. It is conjectured that Bernoulli actions of nonamenable groups with vanishing first $\ell^2$-Betti number should always be $W^*$-superrigid. This conjecture is a consequence of the combination of two conjectures of Popa, see \cite[pages 2 and 4]{Popa-problems} and \cite[Problems III and IV]{Ioana-ICM}. Theorem~\ref{theointro:Bernoulli} provides new classes of groups that verify this conjecture.

\medskip

Let us conclude this introduction by presenting the main steps of our proof of Theorem~\ref{theointro:strong-rigidity}. We have a cocycle $c:G\times X\to H$, given by the stable orbit equivalence of the actions. We first observe that it is enough to find a standard generator $s$ of $G$ such that, after replacing $c$ by a cohomologous cocycle (of the form $c'(g,x)=\varphi(gx)c(g,x)\varphi(x)^{-1}$ for some measurable map $\varphi:X\to H$), the map $c_{|\langle s\rangle\times X}$ is almost everywhere constant. Indeed, a propagation argument, using that $s$ is part of a generating set of $G$ with the property that two consecutive elements commute, then shows that $c$ is cohomologous to a group homomorphism (and likewise for the given cocycle $H\times Y\to G$), from which the conclusion follows. This propagation argument is presented in Section~\ref{sec:commuting-chain}.

The first step towards the above goal is to use the techniques from our previous work \cite{HH2} to ``recognize'' certain natural subgroups of $G$ and $H$ from the orbit equivalence relation coming from their actions. More precisely, we prove that there exist maximal join parabolic subgroups $P\subseteq G$ and $Q\subseteq H$ (i.e.\ decomposing as a nontrivial product), and positive measure Borel subsets $U\subseteq X$ and $V\subseteq Y$, such that after identifying $U$ and $V$ through a measure-scaling isomorphism, the intersections of the $P$-orbits with $U$ coincide with the intersections of the $Q$-orbits with $V$. 

If $P$ and $Q$ have trivial center, then we can directly apply Monod and Shalom's rigidity theorem \cite[Theorem~2.17]{MS} regarding actions of direct products of groups in the class $\mathcal{C}_{\mathrm{reg}}$ to get the desired conclusion. 

The most difficult case is when all maximal join parabolic subgroups of $G$ have nontrivial center. This in fact often happens: for instance, if the underlying graph of $G$ is triangle-free and square-free, then the maximal join parabolic subgroups are exactly the star subgroups, isomorphic to $\mathbb{Z}\times F_n$. In this case, a simple combinatorial argument enables us to find two maximal join parabolic subgroups $P_1,P_2\subseteq G$ with commuting centers. Using techniques from \cite{HH2}, we are able to show that the orbits of the subgroups $P_i$, restricted to some positive measure Borel subset $U$, coincide with the orbits (restricted to some $V$) of two maximal join parabolic subgroups $Q_1,Q_2\subseteq H$ with commuting centers.  As the centers $A_1,A_2$ of $P_1,P_2$ act ergodically (and likewise for the centers $B_1,B_2$ of $Q_1,Q_2$), we can then apply another rigidity theorem due to Monod and Shalom \cite{MS} to derive that for every $i\in\{1,2\}$, the cocycle $c$ is cohomologous to a cocycle $c_i$ that induces a group isomorphism between the quotients $P_i/A_i$ and $Q_i/B_i$. Informally, this means that our cocycle $c_i$ is only controlled \emph{up to an ambiguity in the central direction}. But by comparing the ambiguities given by $c_1$ and $c_2$, we manage to cancel them and prove that $c$ is actually cohomologous to a group homomorphism on $A_i$. As explained above, this is enough to conclude our proof. 

\paragraph*{Acknowledgments.} We thank the referees for their careful reading of our manuscript.

The first named author acknowledges support from the Agence Nationale de la Recherche under Grant ANR-16-CE40-0006 DAGGER. 

\section{Right-angled Artin groups and combinatorial lemmas}

\label{sec:raag}

Given a finite simple graph $\Gamma$, the \emph{right-angled Artin group} $G_\Gamma$ is the group defined by the following presentation: 

\begin{center}
	$G_\Gamma=\langle V\Gamma$\ |\ $[v,w]=1$ if $v$ and $w$ are joined by an edge$\rangle$.
\end{center}

The images in $G_\Gamma$ of the vertices of $\Gamma$ form the \emph{standard generating set} of $G_\Gamma$. A \emph{full subgraph} of $\Gamma$ is a subgraph $\Lambda\subseteq\Gamma$ such that two vertices of $\Lambda$ are adjacent in $\Lambda$ if and only if they are adjacent in $\Gamma$. Any full subgraph $\Lambda\subseteq\Gamma$ induces an injective homomorphism $G_{\Lambda}\hookrightarrow G_{\Gamma}$ (sending the standard generating set of $G_\Lambda$ to a subset of the standard generating set of $G_\Gamma$), whose image is called a \emph{standard subgroup} of $G_{\Gamma}$. Conjugates of standard subgroups are called \emph{parabolic subgroups} of $G_\Gamma$. 

It is known that if $gG_{\Lambda_1}g^{-1}\subseteq G_{\Lambda_2}$ for some full subgraphs $\Lambda_1,\Lambda_2$ of $\Gamma$, then  $\Lambda_1\subseteq\Lambda_2$ and there exists $h\in G_{\Lambda_2}$ such that  $hG_{\Lambda_1}h^{-1}=gG_{\Lambda_1}g^{-1}$ (this follows from \cite[Proposition~2.2]{charney2007automorphisms}). Thus the parabolic subgroup $gG_{\Lambda_1}g^{-1}$ of $G_\Gamma$ is also a parabolic subgroup of $G_{\Lambda_2}$.  

For a full subgraph $\Lambda\subseteq\Gamma$, define $\Lambda^\perp$ to be the full subgraph spanned by all vertices in $V\Gamma\setminus V\Lambda$ that are adjacent to all vertices of $\Lambda$. Let now $P=gG_{\Lambda}g^{-1}$ be a parabolic subgroup. We define $P^\perp=gG_{\Lambda^\perp}g^{-1}$. This is well-defined: if we can write the parabolic subgroup $P$ in two different ways $gG_{\Lambda}g^{-1}$ and $hG_{\Lambda'}h^{-1}$, then \cite[Proposition~2.2]{charney2007automorphisms} implies that $\Lambda=\Lambda'$ and $gG_{\Lambda^\perp}g^{-1}=hG_{\Lambda^\perp}h^{-1}$. 

\begin{lemma}[{Charney--Crisp--Vogtmann \cite[Proposition~2.2]{charney2007automorphisms}}]
	\label{lemma:normalizer}
Let $P\subseteq G_\Gamma$ be a parabolic subgroup. Then the normalizer of $P$ in $G_{\Gamma}$ is $P\times P^\perp$. 
\end{lemma}

 Many properties of $G_\Gamma$ can be read from its defining graph $\Gamma$. For instance $G_\Gamma$ is one-ended if and only if $\Gamma$ is connected, and $G_\Gamma$ has trivial center if and only if no vertex of $\Gamma$ is connected to every other vertex.

For any full subgraph $\Lambda\subseteq\Gamma$, there is a retraction $r_\Lambda:G_{\Gamma}\to G_{\Lambda}$ defined by sending every element of the standard generating set corresponding to a vertex in $V\Gamma\setminus V\Lambda$ to the identity element. Hence for any parabolic subgroup $P=gG_{\Lambda}g^{-1}$ of $G_{\Gamma}$, we have a (uniquely well-defined) retraction $r_P:G_{\Gamma}\to P$, defined by letting $r_P(gsg^{-1})=gr_\Lambda(s)g^{-1}$ for every standard generator $s$ of $G_\Gamma$.

 A \emph{join subgraph} $\Lambda$ of $\Gamma$ is a full subgraph which admits a join decomposition $\Lambda=\Lambda_1\circ \Lambda_2$ (i.e.\ every vertex of $\Lambda_1$ is adjacent to every vertex of $\Lambda_2$) with $\Lambda_i\neq \emptyset$ for every $i\in\{1,2\}$. A \emph{maximal join subgraph} is a join subgraph which is not properly contained in another join subgraph.

A \emph{(maximal) join parabolic subgroup} is a parabolic subgroup of form $gG_{\Lambda}g^{-1}$ where $\Lambda$ is a (maximal) join subgraph of $\Gamma$.

The \emph{clique factor} of a graph $\Lambda$ is the maximal complete subgraph appearing in a join decomposition of $\Lambda$.

\begin{lemma}\label{lemma:join-parabolic}
Let $G=G_\Gamma$ be a right-angled Artin group, let $P$ be a join parabolic subgroup of $G$, and let $S\subseteq P$ be a parabolic subgroup. Then $S\times S^{\perp}$ is a join parabolic subgroup.
\end{lemma}

\begin{proof}
Let $\Lambda\subseteq\Gamma$ be a full subgraph such that $P$ is conjugate to $G_\Lambda$; the subgraph $\Lambda$ decomposes nontrivially as a join $\Lambda=\Lambda_1\circ\Lambda_2$. Then $S$ is conjugate to $G_\Upsilon$ for some full subgraph $\Upsilon$ of $\Lambda$ (as follows from \cite[Proposition~2.2]{charney2007automorphisms}). If $\Upsilon\subseteq\Lambda_i$ for some $i\in\{1,2\}$, then $\Upsilon^{\perp}$ contains $\Lambda_{3-i}$, so $S\times S^{\perp}$ is a join parabolic subgroup. Otherwise $\Upsilon$ decomposes nontrivially as a join, and $S$ itself is a join parabolic subgroup (and therefore so is $S\times S^{\perp}$).
\end{proof}

\begin{lemma}\label{lemma:maximal-not-abelian}
Let $G=G_\Gamma$ be a nonabelian right-angled Artin group with connected defining graph. Then no maximal join parabolic subgroup is abelian.
\end{lemma}

\begin{proof}
Let $\Omega\subseteq\Gamma$ be a maximal join  subgraph, and assume towards a contradiction that $\Omega$ is a clique. As $G$ is nonabelian and $\Gamma$ is connected, we can find a vertex $v\in V\Omega$ which is joined by an edge to a vertex $u\notin V\Omega$. In particular, $v\circ v^{\perp}$ is a join  subgraph of $\Gamma$ which properly contains $\Omega$, contradicting the maximality of $\Omega$.  
\end{proof}

The following basic combinatorial lemma will be crucial for the general structure of the proof of our main theorems: two different arguments will be used in the paper, depending on whether $G_\Gamma$ satisfies the first or second conclusion below.

\begin{lemma}\label{lemma:combinatorics}
Let $G=G_\Gamma$ be a one-ended right-angled Artin group with trivial center. Then either $G$ contains a maximal join parabolic subgroup with trivial center, or else $G$ contains two distinct  nonabelian maximal join parabolic subgroups whose centers commute.
\end{lemma}

\begin{proof}
We assume that every maximal join parabolic subgroup of $G$ has a nontrivial center, and prove that the second conclusion of the lemma holds. Let $\Omega$ be a maximal join subgraph in $\Gamma$, with clique factor $\Omega_1$. As $G$ has trivial center and $\Gamma$ is connected (because $G$ is one-ended), there is a vertex $v\in V\Omega$ such that $v$ is adjacent to a vertex $u$ outside $\Omega$. Let $\Lambda$ be a maximal join subgraph containing $v\circ v^\perp$. Then $\Omega_1\subsetneq v\circ v^\perp\subseteq \Lambda$ and $\Omega\neq\Lambda$ (as $u\in V\Lambda$). By Lemma~\ref{lemma:maximal-not-abelian}, the parabolic subgroups $G_\Omega$ and $G_\Lambda$ are nonabelian. Finally, letting $\Lambda_1$ be the clique factor of $\Lambda$, the group $G_{\Lambda_1}$ commutes with $G_{v\circ v^{\perp}}$, in particular $G_{\Lambda_1}$ and $G_{\Omega_1}$ commute.
\end{proof}

\begin{lemma}\label{lemma:commuting-centers}
Let $G=G_{\Gamma}$ be a right-angled Artin group, and let $P_1,P_2\subseteq G$ be two distinct maximal join parabolic subgroups. For every $i\in\{1,2\}$, let $Z_i$ be the center of $P_i$.

Then $Z_1\cap Z_2=\{1\}$. In particular, if $Z_1$ and $Z_2$ commute, then $Z_1\subseteq Z_2^{\perp}$ and $Z_2\subseteq Z_1^{\perp}$.
\end{lemma}

\begin{proof}
For every $i\in\{1,2\}$, the subgroup $Z_i$ is a parabolic subgroup of $G$, so $Z_1\cap Z_2$ is a parabolic subgroup of $G$ by \cite[Proposition~2.6]{DKR}. Let $Z=Z_1\cap Z_2$, and assume towards a contradiction that $Z\neq\{1\}$. Then $P=Z\times Z^{\perp}$ is a join parabolic subgroup of $G$ which contains $P_1$ and $P_2$. By maximality, we have $P_1=P_2=P$, a contradiction.  

We will now prove the last assertion of the lemma, so assume that $Z_1$ and $Z_2$ commute. Then $Z_2$ is a parabolic subgroup of $G$ contained in $Z_1\times Z_1^{\perp}$, so it is a parabolic subgroup of $Z_1\times Z_1^{\perp}$ (as can be derived from \cite[Proposition~2.2(2)]{charney2007automorphisms}). But \cite[Proposition~2.2(2)]{charney2007automorphisms} also ensures that parabolic subgroups of $Z_1\times Z_1^{\perp}$ are of the form $A\times B$, where $A$ is a parabolic subgroup of $Z_1$ and $B$ is a parabolic subgroup of $Z_1^{\perp}$. As $Z_2\cap Z_1=\{1\}$, it follows that $Z_2\subseteq Z_1^{\perp}$. The fact that $Z_2\subseteq Z_1^{\perp}$ follows by symmetry.
\end{proof}

Recall that a countable group $G$ is \emph{ICC} (standing for \emph{infinite conjugacy classes}) if the conjugacy class of every nontrivial element of $G$ is infinite.

\begin{lemma}\label{lemma:icc}
Every right-angled Artin group with trivial center is ICC.
\end{lemma}

\begin{proof}
Let $G$ be a right-angled Artin group with trivial center,  with defining graph $\Gamma$, and let $\Gamma=\Gamma_1\circ \cdots\circ \Gamma_{k}$ be a join decomposition of $\Gamma$ into factors which does not allow any further non-trivial join decomposition. Then each $G_{\Gamma_i}$ has trivial center. It suffices to prove that each $G_{\Gamma_i}$ is ICC. Note that $G_{\Gamma_i}$ is acylindrically hyperbolic in the sense of \cite{Osi}: the case when $\Gamma_i$ is connected follows from \cite[Theorem~30]{KK}, or alternatively from the combination of \cite{Sis} and \cite{Osi}, and the case when $\Gamma_i$ is disconnected follows from the fact that $G_{\Gamma_i}$ splits non-trivially as a free product. Hence $G_{\Gamma_i}$ is ICC by \cite[Theorem~2.35]{DGO}.
\end{proof}

\section{Background on stable orbit equivalence and measured groupoids}

This section reviews material regarding stable orbit equivalence, cocycles and measured groupoids. A familiar reader can directly skip to the next section.

\subsection{Stable orbit equivalence and cocycles}\label{sec:soe}

A \emph{standard Borel space} is a measurable space $X$ which is isomorphic to a Polish topological space (i.e.\ separable and completely metrizable) equipped with its Borel $\sigma$-algebra. By a \emph{standard probability space} we mean a standard Borel space equipped with a Borel measure $\mu$ such that $\mu(X)=1$. In this paper, all actions of countable groups on standard Borel spaces are assumed to be by Borel automorphisms. Given a standard probability space $(X,\mu)$ and a Borel subset $A\subseteq X$ of positive measure, we denote by $\mu_A$ the Borel probability measure on $A$ defined by renormalizing $\mu_{|A}$.

Let $G$ and $H$ be two countable groups, and assume we have a measure-preserving $G$-action on a standard probability space $X$. A measurable map $c:G\times X\to H$ is a \emph{cocycle} if for every $g,g'\in G$ and almost every $x\in X$, one has $c(gg',x)=c(g,g'x)c(g',x)$. The cocycle $c$ is \emph{strict} if this relation holds for all $g,g'\in G$ and \emph{all} $x\in X$. As $G$ is countable, there always exists a $G$-invariant conull Borel subset $X^*\subseteq X$ such that $c_{|G\times X^*}$ is a strict cocycle. Two cocycles $c,c':G\times X\to H$ are \emph{cohomologous} if there exists a measurable map $\varphi:X\to H$ such that for all $g\in G$ and almost every $x\in X$, one has $c'(g,x)=\varphi(gx)c(g,x)\varphi(x)^{-1}$.

We now briefly review the notion of \emph{stably orbit equivalent} group actions, and refer the reader to \cite{Fur-oe} for more information. Let $G\actson (X,\mu)$ and $H\actson (Y,\nu)$ be two free, ergodic, measure-preserving actions on standard probability spaces. A \emph{stable orbit equivalence} between $G\actson X$ and $H\actson Y$ is a measure space isomorphism $f:(U,\mu_U)\to (V,\nu_V)$, where $U\subseteq X$ and $V\subseteq Y$ are positive measure Borel subsets, such that $f((G\cdot x)\cap U)=(H\cdot f(x))\cap V$ for almost every $x\in U$. The \emph{compression constant} of $f$ is defined as $\kappa(f)=\nu(V)/\mu(U)$. Following the exposition from \cite[Section~4]{Vae}, we say that a cocycle $c:G\times X\to H$ is an \emph{SOE cocycle associated to $f$} if there exists a measurable map $p:X\to U$, with $p(x)\in G\cdot x$ for almost every $x\in X$, such that for almost every $x\in X$, $c(g,x)$ is the unique element $h\in H$ such that $f\circ p(g\cdot x)=h \cdot (f\circ p(x))$ (uniqueness comes from freeness of the $H$-action). An SOE cocycle associated to $f$ always exists by ergodicity of the $G$-action (i.e.\ we can always find a map $p$ as above), and any two such cocycles (corresponding to different choices of $p$) are cohomologous. Notice that we can always choose $p$ as above such that $p_{|U}=\mathrm{id}_U$. The two actions $G\actson X$ and $H\actson Y$ are \emph{stably orbit equivalent} if there exists a stable orbit equivalence between them; they are \emph{orbit equivalent} if it can be chosen with $U=X$ and $V=Y$. We mention that two free, ergodic, measure-preserving actions on standard probability spaces are orbit equivalent if and only if there is a stable orbit equivalence between them whose compression constant is equal to $1$, see \cite[Proposition~2.7]{Fur-oe}.

In the above situation, observe that if $A\subseteq G$ and $B\subseteq H$ are subgroups acting ergodically on $X,Y$, and satisfy $f((A\cdot x)\cap U)=(B\cdot f(x))\cap V$, then an SOE cocycle associated to $f$ can always be chosen so that $c_{|A\times X}$ is an SOE cocycle associated to $f$, viewed as a stable orbit equivalence between the actions $A\actson X$ and $B\actson Y$ (in particular $c(A\times X^*)\subseteq B$ for some conull Borel subset $X^*\subseteq X$). Indeed, this is proved by choosing the map $p$ so that $p(x)\in A\cdot x$ for almost every $x\in X$.

\subsection{Background on measured groupoids}

The arguments in Section~\ref{sec:recognition} below rely on earlier work of the first two named authors \cite{HH2}, which is phrased in the language of measured groupoids. In this section we offer a quick review, and refer the reader to \cite[Section~2.1]{AD}, \cite{Kid-survey} or \cite[Section~3]{GH} for more detailed treatments. It is possible to skip this section for now and come back to it when reading Section~\ref{sec:recognition}. 

A \emph{discrete Borel groupoid} is a standard Borel space $\calg$ equipped with two Borel maps $s,r:\calg\to X$ towards a standard Borel space $X$ whose fibers are at most countable, and coming with a measurable (partially defined) composition law, a measurable inverse map, and a unit element $e_x$ per $x\in X$. The space $X$ is called the \emph{base space} of the groupoid, and we think of an element $g\in\calg$ as being an arrow whose source $s(g)$ and range $r(g)$ both belong to $X$ (composition of two arrows $g_1g_2$ makes sense when $s(g_1)=r(g_2)$). A \emph{bisection} of $\calg$ is a Borel subset $B\subseteq\calg$ such that $s_{|B}$ and $r_{|B}$ are injective; it thus defines a Borel isomorphism between two Borel subsets of $X$ (see \cite[Corollary~15.2]{Kec}). A theorem of Lusin and Novikov (see \cite[Theorem~18.10]{Kec}) ensures that any discrete Borel groupoid is covered by countably many pairwise disjoint bisections. A \emph{measured groupoid} is a discrete Borel groupoid $\calg$ whose base space $X$ comes equipped with a \emph{quasi-invariant} finite Borel measure $\mu$, i.e.\ for every bisection $B\subseteq\calg$, one has $\mu(s(B))=0$ if and only if $\mu(r(B))=0$. A measured groupoid $\calg$ is \emph{trivial} if $\calg=\{e_x|x\in X\}$. On the other hand $\calg$ is \emph{of infinite type} if for every Borel subset $U\subseteq X$ of positive measure, and almost every $x\in U$, there are infinitely many elements $g\in\calg$ with $s(g)=x$ and $r(g)\in U$.

In the present paper, the most important example of a measured groupoid is the following. Let $G$ be a countable group which acts on a standard finite measure space $X$ by Borel automorphisms in a measure-preserving way (or merely by preserving the measure class). Then $G\times X$ is naturally a measured groupoid over $X$, with $s(g,x)=x$ and $r(g,x)=gx$. This groupoid is denoted by $G\ltimes X$.

Let now $\calg, X$ and $\mu$ be as above. Every Borel subset $\calh\subseteq\calg$ which is stable under composition and inversion, and contains all unit elements $e_x$, has the structure of a discrete Borel groupoid over $X$, for which $\mu$ is quasi-invariant; we say that $\calh$ is a \emph{measured subgroupoid} of $\calg$. Given two measured subgroupoids $\calh_1,\calh_2\subseteq\calg$, we denote by $\langle\calh_1,\calh_2\rangle$ the subgroupoid generated by $\calh_1$ and $\calh_2$, defined as the smallest measured subgroupoid of $\calg$ that contains $\calh_1$ and $\calh_2$; equivalently, this is the measured subgroupoid of $\calg$ made of all elements that are finite compositions of elements of $\calh_1$ and $\calh_2$. Given any Borel subset $U\subseteq X$, the \emph{restriction} $\calg_{|U}=\{g\in\calg|s(g),r(g)\in U\}$ is naturally a measured groupoid over $U$, with quasi-invariant measure $\mu_{|U}$. 

Given a countable group $G$, a \emph{strict cocycle} $\rho:\calg\to G$ is a Borel map such that for all $g_1,g_2\in\calg$ satisfying $s(g_1)=r(g_2)$ (so that $g_1g_2$ is well-defined), one has $\rho(g_1g_2)=\rho(g_1)\rho(g_2)$. Its \emph{kernel} is $\{g\in\calg|\rho(g)=1\}$, a measured subgroupoid of $\calg$. A strict cocycle $\rho:\calg\to G$ is \emph{action-type} (as in \cite[Definition~3.20]{GH}) if it has trivial kernel, and for every infinite subgroup $H\subseteq G$, the subgroupoid $\rho^{-1}(H)$ is of infinite type. The following example is crucial: if $G$ acts on a standard finite measure space $X$ by Borel automorphisms in a measure-preserving way, then the natural cocycle $G\ltimes X\to G$ is action-type \cite[Proposition~2.26]{Kid-survey}. 

Let now $\calh$ and $\calh'$ be two measured subgroupoids of $\calg$. The subgroupoid $\calh'$ is \emph{stably contained} in $\calh$ (resp.\ \emph{stably equal} to $\calh$) if there exist a conull Borel subset $X^*\subseteq X$ and a partition $X^*=\dunion_{i\in I}X_i$ into at most countably many Borel subsets such that for every $i\in I$, one has $\calh'_{|X_i}\subseteq\calh_{|X_i}$ (resp.\ $\calh'_{|X_i}=\calh_{|X_i}$).

The subgroupoid $\calh$ is \emph{normalized} by $\calh'$ if there exists a conull Borel subset $X^*\subseteq X$ such that $\calh'_{|X^*}$ can be covered by at most countably many bisections $B_n$ in such a way that for every $n$, every $g_1,g_2\in B_n$, and every $h\in\calh'_{|X^*}$ such that $g_2hg_1^{-1}$ is well-defined, one has $h\in\calh$ if and only if $g_2hg_1^{-1}\in\calh$. Here is an example: if $\calg$ comes equipped with a cocycle $\rho:\calg\to G$ towards a countable group, and if $H,H'\subseteq G$ are two subgroups such that $H$ is normalized by $H'$, then $\rho^{-1}(H)$ is normalized by $\rho^{-1}(H')$. The subgroupoid $\calh$ is \emph{stably normalized} by $\calh'$ if there exists a partition $X=\dunion_{i\in I}X_i$ into at most countably many Borel subsets such that for every $i\in\mathbb{N}$, the groupoid $\calh_{|X_i}$ is normalized by $\calh'_{|X_i}$.

We refer to \cite{Kid-survey} for the notion of \emph{amenability} of a measured groupoid, and only record a few properties we will need. Amenability of measured groupoids is stable under passing to subgroupoids and taking restrictions, and under stabilization in the following sense: if there exist a conull Borel subset $X^*\subseteq X$ and a partition $X^*=\dunion_{i\in I}X_i$ into at most countably many Borel subsets such that for every $i\in I$, the groupoid $\calg_{|X_i}$ is amenable, then $\calg$ is amenable (see \cite[Definition~3.33 and Remark~3.34]{GH}). If $\rho:\calg\to G$ is a strict cocycle with trivial kernel towards a countable group $G$, and if $A\subseteq G$ is amenable, then $\rho^{-1}(A)$ is amenable (see e.g.\ \cite[Corollary~3.39]{GH}). 

A measured groupoid $\calg$ over a standard finite measure space $X$ is \emph{everywhere nonamenable} if for every Borel subset $U\subseteq X$ of positive measure, the restricted groupoid $\calg_{|U}$ is nonamenable. The following fact is crucial: if $\rho:\calg\to G$ is a strict action-type cocycle towards a countable group $G$, and if $G$ contains a nonabelian free subgroup, then $\calg$ is everywhere nonamenable \cite[Lemma~3.20]{Kid} (compare also \cite[Remark~3.3]{HH2}).

\section{Monod and Shalom's rigidity theorems}

\subsection{Quotient by a normal subgroup}

The following lemma is extracted from the work of Monod and Shalom \cite{MS}. Its proof comes from \cite[p.862]{MS}; we recall it here for the convenience of the reader.

\begin{lemma}[Monod--Shalom \cite{MS}]\label{lem:MS}
	Let $G,H$ be countable groups, and let $G\actson X$ and $H\actson Y$ be free ergodic measure-preserving actions on standard probability spaces. Assume that they are stably orbit equivalent, let $f:U\to V$ be a stable orbit equivalence between them (where $U\subseteq X$ and $V\subseteq Y$ are positive measure Borel subsets), and let $c:G\times X\to H$ be an SOE cocycle associated to $f$.

	Let $A\unlhd G$ and $B\unlhd H$ be normal subgroups acting ergodically on $X,Y$, and assume that for every $x\in U$, one has $f((A\cdot x)\cap U)= (B\cdot f(x))\cap V$. 

Then there exist a group isomorphism $\alpha:G/A\to H/B$ and a measurable map $\varphi:X\to H$ with $\varphi(x)=e$ for every $x\in U$, such that for every $g\in G$ and almost every $x\in X$, one has $\varphi(gx) c(g,x)\varphi(x)^{-1}\in\alpha(gA)$. 
\end{lemma}

\begin{proof}
As observed in Section~\ref{sec:soe}, up to replacing $c$ by a cohomologous cocycle, and $X$ by a conull $G$-invariant Borel subset, we can (and will) assume that $c(A\times X)\subseteq B$. Likewise, up to replacing $Y$ by a conull $H$-invariant subset, we can choose an SOE cocycle $c':H\times Y\to G$ associated to the stable orbit equivalence $f^{-1}:V\to U$ between $H\actson Y$ and $G\actson X$, so that $c'(B\times Y)\subseteq A$.

Let $\Sigma=X\times H$, equipped with the measure-preserving action of $G\times H$ given by $(g,h)\cdot (x,k)= (gx,c(g,x)kh^{-1})$. Letting $X_e=X\times\{e\}$ (which is a fundamental domain for the $H$-action on $\Sigma$), we observe that $A X_e\subseteq B X_e$, so $A B X_e=BAX_e\subseteq B BX_e= BX_e$, thus $BX_e$ is invariant under $A\times B$. In addition, from the ergodicity of the $A$-action on $X$, we deduce that the action of $A\times B$ on $BX_e$ is ergodic. Moreover, for every $h\in H$, we have $hBX_e=BX_e$ if and only if $h\in B$, and otherwise $hBX_e\cap BX_e=\emptyset$. In addition, the union of all $H$-translates of $BX_e$ cover $\Sigma$. This proves that $\bar H=H/B$ acts simply transitively on the space $\bar\Sigma$ of ergodic components of the action of $A\times B$ on $\Sigma$. 

By \cite[Theorem~3.3]{Fur-oe}, the space $\Sigma$ is measurably isomorphic to $Y\times G$, equipped with the measure-preserving action of $G\times H$ given by $(g,h)\cdot (y,k)=(hy,c'(h,y)kg^{-1})$. A symmetric argument then shows that $\bar G=G/A$ also acts simply transitively on $\bar \Sigma$.

Therefore, there exist an isomorphism $\alpha:\bar G\to \bar H$, and a measurable isomorphism $\bar{\Sigma}\approx\bar H$ sending $BX_e$ to $e$, such that the action of $\bar G\times\bar H$ on $\bar{\Sigma}$ is given by $(\bar g,\bar h)\cdot\bar{k}=\alpha(\bar g)\bar k\bar h^{-1}$ through this identification. We also have a $(G\times H)$-equivariant Borel map $\Phi:\Sigma\to\bar H$ (sending $BX_e$ to $e$). 

The equivariance of $\Phi$ shows that for all $g\in G$ and almost every $x\in X$, one has $\Phi(g(x,e))=\alpha(\bar g)$, i.e.\ $\Phi(gx,c(g,x))=\alpha(\bar g)$. Letting $h\in H$ be such that $\alpha(\bar g)=\bar h$, we deduce that $\Phi(gx,c(g,x)h^{-1})=e$. This shows that $c(g,x)h^{-1}\in B$, i.e.\ $c(g,x)\in \alpha(gA)$, as desired.
\end{proof}

\subsection{Direct products} 

Following \cite[Notation~1.2]{MS}, we let $\Creg$ be the class of all countable groups $\Gamma$ such that $\bddc(\Gamma,\ell^2(\Gamma))\neq 0$. By \cite[Corollary~1.8]{CFI}, every nonabelian right-angled Artin group which does not split nontrivially as a direct product belongs to the class $\Creg$ (this also follows from \cite{Ham,HO} and the fact that these groups are acylindrically hyperbolic, see \cite{KK} or \cite{Sis,Osi}).

\begin{theo}[{Monod--Shalom \cite[Theorem~2.17]{MS}}]\label{theo:monod-shalom-product}
Let $m,n\ge 2$, and let $G_1,\dots,G_m$ and $H_1,\dots,H_n$ be torsion-free countable groups in $\Creg$. Let $G=G_1\times\dots\times G_m$ and $H=H_1\times\dots\times H_n$. 

Let $G\actson X$ and $H\actson Y$ be two free, ergodic, measure-preserving actions on standard probability spaces. Assume that all groups $G_i$ act ergodically on $X$, and all groups $H_j$ act ergodically on $Y$, and the actions are stably orbit equivalent (via a stable orbit equivalence $f:U\to V$). 

Then $\kappa(f)=1$, the actions $G\actson X$ and $H\actson Y$ are conjugate through a group isomorphism between $G$ and $H$, and every SOE cocycle $G\times X\to H$ is cohomologous to a group isomorphism.
\end{theo}

\section{Exploiting chain-commuting generating sets}\label{sec:commuting-chain}

We start with an elementary lemma.

\begin{lemma}\label{lemma:cocycle-homomorphism}
Let $G$ and $H$ be groups, with $G$ countable. Let $G\actson X$ be a measure-preserving $G$-action on a standard probability space $X$, and let $c:G\times X\to H$ be a cocycle. Let $S\subseteq G$ be a generating set for $G$. Assume that there exists a conull Borel subset $X^*\subseteq X$ such that for every $s\in S$, the value of $c(s,\cdot)_{|X^*}$ is constant.

Then there exists a group homomorphism $\alpha:G\to H$ such that for every $g\in G$ and almost every $x\in X$, one has $c(g,x)=\alpha(g)$.
\end{lemma}

\begin{proof}
The fact that $c(g,\cdot)$ is almost everywhere constant follows from the same fact for $g\in S$ together with our assumption that $S$ generates $G$. Letting $\alpha:G\to H$ be defined by sending $g$ to the essential value of $c(g,\cdot)$, the fact that $c$ is a cocycle implies that $\alpha$ is a homomorphism, completing the proof.
\end{proof}

A generating set $S$ of a group $G$ is \emph{chain-commuting} if the graph whose vertex set is $S$, with one edge between two vertices if the corresponding elements of $G$ commute, is connected. Notice that a group $G$ has a finite chain-commuting generating set if and only if it is a quotient of a one-ended right-angled Artin group (defined over a finite simple graph $\Gamma$). Interestingly, having a finite chain-commuting generating set whose elements have infinite order is a condition that has already been successfully exploited in various contexts in measured group theory: for instance Gaboriau proved in \cite[Critères VI.24]{Gab-cost} that it forces all free probability measure-preserving actions of $G$ to have cost $1$; see also \cite{AGN} for a more recent use. 

We say that a group $H$ has the \emph{root-conjugation property} if for every $h_1,h_2\in H$ and every integer $k>0$, if $h_1$ commutes with $h_2^k$, then $h_1$ commutes with $h_2$. 

\begin{lemma}\label{lemma:chain-commutative}
Let $G$ and $H$ be countable groups. Assume that $H$ satisfies the root-conjugation property. Let $G\actson X$ be a measure-preserving $G$-action on a standard probability space $X$, and let $c:G\times X\to H$ be a cocycle. Let $S$ be a generating set of $G$. Assume that 
\begin{enumerate}
\item $S$ is chain-commuting,
\item every element of $S$ acts ergodically on $X$, and
\item there exist $s\in S$ and a conull Borel subset $X^*\subseteq X$ such that $c(s,\cdot)_{|X^*}$ is constant.
\end{enumerate}
Then there exists a group homomorphism $\alpha:G\to H$ such that for every $g\in G$ and almost every $x\in X$, one has $c(g,x)=\alpha(g)$.
\end{lemma}

\begin{proof}
Let $s\in S$ be as in assertion~3, and denote by $\beta_s$ the constant value of $c(s,\cdot)$ on $X^*$. We claim that for every $u\in S$ which commutes with $s$, the value $c(u,\cdot)$ is constant on a conull Borel subset of $X$. As $S$ is chain-commuting, arguing inductively will then ensure that the same is true for all $u\in S$, and as $G$ is countable the conull Borel subset of $X$ can be chosen independent of $u$. The conclusion will then follow from Lemma~\ref{lemma:cocycle-homomorphism}.

We now prove the above claim. Up to replacing $X^*$ by a further conull Borel subset (which we can assume to be $G$-invariant), we will assume that the cocycle $c$ is strict. Let $X^*=\sqcup_{i\in I}X_i$ be a partition into at most countably many Borel subsets such that for each $i$, the value of $c(u,\cdot)$ is constant when restricted to $X_i$ -- we denote it by $\alpha_i$. Let $i,j\in I$ be such that $X_i$ and $X_j$ have positive measure (possibly with $i=j$). As $s$ acts ergodically on $X$, there exist an integer $k_{i,j}\neq 0$ and $x\in X_i$ such that $s^{k_{i,j}}x\in X_j$. As $u$ and $s^{k_{i,j}}$ commute, we have $c(us^{k_{i,j}},x)=c(s^{k_{i,j}}u,x)$. Thus $c(u,s^{k_{i,j}}x)c(s^{k_{i,j}},x)=c(s^{k_{i,j}},ux)c(u,x)$, in other words 
\begin{equation}\label{eq:commute}
\alpha_j\beta_s^{k_{i,j}}=\beta_s^{k_{i,j}}\alpha_i.
\end{equation} 
Letting $i=j$, we see that $\alpha_i$ commutes with $\beta_s^{k_{i,j}}$. By the root-conjugation property, it follows that $\alpha_i$ and $\beta_s$ commute. Using Equation~\eqref{eq:commute} again with $i,j$ arbitrary, we see that $\alpha_i=\alpha_j$ whenever both $X_i$ and $X_j$ have positive measure. In other words, the value $c(u,\cdot)$ is almost everywhere constant.
\end{proof}

In the present paper, Lemma~\ref{lemma:chain-commutative} will be applied to the setting of right-angled Artin groups in the following way.

\begin{lemma}\label{lemma:exploting-commuting-gensets}
Let $G,H$ be two right-angled Artin groups, with $G$ one-ended. Let $G\actson X$ and $H\actson Y$ be two free, ergodic, measure-preserving actions on standard probability spaces, and assume that there is a stable orbit equivalence $f$ between $G\actson X$ and $H\actson Y$, with compression constant $\kappa(f)\ge 1$. Assume that $G\actson X$ is irreducible, and let $S$ be a standard generating set of $G$ (given by an isomorphism to some $G_\Gamma$) such that all elements of $S$ act ergodically on $X$. Let $c:G\times X\to H$ be an SOE cocycle associated to $f$.

If $c$ is cohomologous to a cocycle $c'$ for which there exists $s\in S$ such that $c'(s,\cdot)$ is almost everywhere constant, then $\kappa(f)=1$, the cocycle $c$ is cohomologous to a group isomorphism $\alpha:G\to H$, and the actions are conjugate through $\alpha$. 
\end{lemma}

\begin{proof}
Right-angled Artin groups have the root-conjugation property, as follows from \cite[Lemma~6.3]{Min}. In addition, the standard generating set $S$ of $G$ is chain-commuting (because $G$ is one-ended, i.e.\ its defining graph $\Gamma$ is connected), and by assumption every element of $S$ acts ergodically on $X$. Lemma~\ref{lemma:chain-commutative} therefore implies that $c$ is cohomologous to a group homomorphism $\alpha:G\to H$. As $G$ is torsion-free, it follows from \cite[Lemma~4.7]{Vae} that $\alpha$ is injective, $\alpha(G)$ has finite index in $H$, the action $H\actson Y$ is conjugate to the action induced from $G\actson X$, and $\kappa(f)=\frac{1}{[H:\alpha(G)]}$. As $\kappa(f)\ge 1$, we deduce that $[H:\alpha(G)]=1$ (in particular $\alpha$ is a group isomorphism) and \cite[Lemma~4.7]{Vae} ensures that the actions $G\actson X$ and $H\actson Y$ are conjugate through $\alpha$.  
\end{proof}

\section{Recognition lemmas}\label{sec:recognition}

\subsection{Review of parabolic supports}

The following notion was introduced in \cite[Section~3.3]{HH2}. Let $\calg$ be a measured groupoid over a standard finite measure space $X$, equipped with a strict cocycle $\rho:\calg\to G$, where $G$ is a right-angled Artin group. Fix an identification $G=G_\Gamma$, and let $\mathbb{P}$ be the set of all parabolic subgroups of $G$ with respect to this identification. Given $P\in\mathbb{P}$, we say that $(\calg,\rho)$ is \emph{tightly $P$-supported} if 
\begin{enumerate}
\item there exists a conull Borel subset $X^*\subseteq X$ such that $\rho(\calg_{|X^*})\subseteq P$, and
\item for every parabolic subgroup $Q\subsetneq P$ and every Borel subset $U\subseteq X$ of positive measure, one has $\rho(\calg_{|U})\nsubseteq Q$.
\end{enumerate}
A parabolic subgroup $P$ such that $(\calg,\rho)$ is tightly $P$-supported, if it exists, is unique. The following lemma records the contents of \cite[Lemma~3.7 and Remark~3.9]{HH2}. 

\begin{lemma}\label{lemma:support-exists}
Let $G=G_\Gamma$ be a right-angled Artin group, let $\calg$ be a measured groupoid over a standard finite measure space $X$, and let $\rho:\calg\to G$ be a strict cocycle.

Then there exists a partition $X=\dunion_{i\in I}X_i$ into at most countably many Borel subsets, and for every $i\in I$, a parabolic subgroup $P_i$, such that $(\calg_{|X_i},\rho)$ is tightly $P_i$-supported. 
\end{lemma}

The following is a consequence of Lemma~\ref{lemma:normalizer} and \cite[Lemma~3.8 and Remark~3.9]{HH2}.

\begin{lemma}\label{lemma:support-invariant-normal}
	Let $G=G_\Gamma$ be a right-angled Artin group, let $\calg$ be a measured groupoid over a standard finite measure space $X$, and let $\rho:\calg\to G$ be a strict cocycle. Let $\calh$ and $\calh'$ be two measured subgroupoids of $\calg$. Assume that $(\calh,\rho)$ is tightly $P$-supported for a parabolic subgroup $P$. Assume also that $\calh$ is normalized by $\calh'$.

	Then there exists a conull Borel subset $X^*\subseteq X$ such that $\rho(\calh'_{|X^*})\subseteq P\times P^\perp$. \qed
\end{lemma}

We now establish a lemma which essentially follows from \cite{HH2}.

\begin{lemma}\label{lemma:support-of-normal-amenable}
	Let $G=G_\Gamma$ be a right-angled Artin group. Let $\calg$ be a measured groupoid over a standard finite measure space $X$ and let $\rho:\calg\to G$ be a strict cocycle with trivial kernel. Let $\calh$ be a measured subgroupoid of $\calg$. Assume that $\calh$ is everywhere nonamenable and stably normalizes an amenable subgroupoid $\cala$.

	Then there exist a conull Borel subset $X^*\subseteq X$, a partition $X^*=\dunion_{i\in I}X_i$ into at most countably many Borel subsets of positive measure, and for every $i\in I$, a parabolic subgroup $P_i$ such that
	\begin{enumerate}
		\item $\cala_{|X_i}\subseteq\rho^{-1}(P_i)_{|X_i}$,
		\item $\calh_{|X_i}\subseteq\rho^{-1}(P_i\times P_i^{\perp})_{|X_i}$,
		\item $(\calh\cap\rho^{-1}(P_i))_{|X_i}$ is amenable,
		\item $P_i^{\perp}$ is nonabelian.
	\end{enumerate} 
\end{lemma}

\begin{proof}
 Consider a partition $X=\dunion_{i\in I}X_i$ into at most countably many Borel subsets such that for every $i\in I$, there exists a parabolic subgroup $P_i$ such that $(\cala_{|X_i},\rho)$ is tightly $P_i$-supported  (given by Lemma~\ref{lemma:support-exists}). As $\cala$ is stably normalized by $\calh$, Lemma~\ref{lemma:support-invariant-normal} ensures that up to replacing $X$ by a conull Borel subset and refining the above partition, we can assume that $\calh_{|X_i}\subseteq\rho^{-1}(P_i\times P_i^{\perp})_{|X_i}$ for every $i\in I$. As $\cala$ is stably normalized by $\calh$ which is everywhere nonamenable, and as $\rho$ has trivial kernel, it follows from \cite[Lemma~3.10]{HH2} that $P_i^\perp$ is nonamenable, and $(\calh\cap\rho^{-1}(P_i))_{|X_i}$ is amenable. 
\end{proof}   

\subsection{Recognizing maximal join parabolic subgroupoids}

Given an equivalence relation arising from a probability measure-preserving action of a right-angled Artin group, the following lemma will enable us to recognize subrelations arising from restricting the action to a maximal join parabolic subgroup. Its proof is based on the techniques developed in \cite{HH2}. 

\begin{lemma}\label{lemma:recognize-product}
	Let $G$ be a one-ended nonabelian right-angled Artin group. Let $\calg$ be a measured groupoid over a standard finite measure space $X$, coming with a strict action-type cocycle $\rho:\calg\to G$. Let $\calh$ be a measured subgroupoid of $\calg$. Then the following assertions are equivalent.
	\begin{enumerate}
		\item There exist a conull Borel subset $X^*\subseteq X$ and a partition $X^*=\dunion_{i\in I}X_i$ into at most countably many Borel subsets such that for every $i\in I$, there exists a maximal join parabolic subgroup $P_i$ such that $\calh_{|X_i}=\rho^{-1}(P_i)_{|X_i}$.
		\item The following properties hold:
		\begin{enumerate}
			\item The subgroupoid $\calh$ contains two subgroupoids $\cala,\caln$, where $\cala$ is amenable, of infinite type, and stably normalized by $\caln$, and $\caln$ is everywhere nonamenable and stably normalized by $\calh$.
			\item Whenever $\calh'$ is another measured subgroupoid of $\calg$ satisfying property~(a), if $\calh$ is stably contained in $\calh'$, then they are stably equal.
		\end{enumerate}
	\end{enumerate}
\end{lemma}

\begin{proof}
In this proof, we fix an identification between $G$ and $G_\Gamma$; parabolic subgroups of $G$ are understood with respect to this identification.

We first prove that $(1)$ implies $(2a)$. For every $i\in I$, the group $P_i$ is nonabelian (Lemma~\ref{lemma:maximal-not-abelian}); hence $P_i$ splits as a direct product $P_i=M_i\times N_i$, where $M_i$ and $N_i$ are infinite parabolic subgroups, and at least one of them (say $N_i$) is nonabelian.  Therefore $N_i$ contains a nonabelian free subgroup. Choose an infinite cyclic subgroup $A_i\subseteq M_i$. Then $A_i$ commutes with $N_i$. The conclusion follows by letting $\cala$ be a measured subgroupoid of $\calg$ such that for every $i\in I$, one has $\cala_{|X_i}=\rho^{-1}(A_i)_{|X_i}$, and letting $\caln$ be such that for every $i\in I$, one has $\caln_{|X_i}=\rho^{-1}(N_i)_{|X_i}$.

We now claim that if a measured subgroupoid $\calh\subseteq\calg$ satisfies $(2a)$, then there exist a conull Borel subset $X^*\subseteq X$, a Borel partition $X^*=\dunion_{i\in I}X_i$ into at most countably many Borel subsets, and for every $i\in I$, a join parabolic subgroup $P_i\subseteq G$ such that $\calh_{|X_i}\subseteq\rho^{-1}(P_i)_{|X_i}$. Once we prove the claim we will explain in the last paragraphs why this suffices to establish the lemma.

By Lemma~\ref{lemma:support-of-normal-amenable}, there exists a conull Borel subset $X^*\subseteq X$ and a partition $X^*=\dunion_{i\in I}X_i$ into at most countably many Borel subsets such that for every $i\in I$, there exists a parabolic subgroup $R_i$ with $R_i^{\perp}$ nonabelian, such that  $\rho(\cala_{|X_i})\subseteq R_i$ and $\rho(\caln_{|X_i})\subseteq R_i\times R_i^\perp$. Notice that $R_i$ is nontrivial because $\cala$ is of infinite type and $\rho$ has trivial kernel. In particular $R_i\times R_i^{\perp}$ is a join parabolic subgroup. 

Up to a further partition, for every $i\in I$, there exists a nontrivial parabolic subgroup $S_i\subseteq R_i\times R_i^{\perp}$ such that $(\caln_{|X_i},\rho)$ is tightly $S_i$-supported (Lemma~\ref{lemma:support-exists}). As $\caln$ is stably normalized by $\calh$, up to a further partition and restriction to a further conull Borel subset of $X$, we can assume that $\rho(\calh_{|X_i})\subseteq S_i\times S_i^{\perp}$ by Lemma~\ref{lemma:support-invariant-normal}. Lemma~\ref{lemma:join-parabolic} ensures that $S_i\times S_i^{\perp}$ is a join parabolic subgroup, which proves our claim.

We have already proved that $(1)$ implies $(2a)$. To see that $(1)$ implies $(2b)$, let $\calh$ be a measured subgroupoid as in $(1)$ (coming with a partition $X^*=\dunion_{i\in I}X_i$ and  maximal join parabolic subgroups $P_i$), and  let $\calh'$ be as in $(2b)$. The above claim ensures that up to passing to a further conull Borel subset and refining the above partition, we can assume that for every $i\in I$, there exists a join parabolic subgroup $Q_i$ such that $\calh'_{|X_i}\subseteq \rho^{-1}(Q_i)_{|X_i}$. As $\calh$ is stably contained in $\calh'$ and $\rho$ is action-type, we deduce that every element of $P_i$ has a power contained in $Q_i$, and therefore $P_i\subseteq Q_i$ by \cite[Lemma~6.4]{Min}. By maximality of $P_i$, we have $P_i=Q_i$, from which it follows that $\calh'$ is stably contained in $\calh$, proving $(2b)$. 

We finally prove that $(2)$ implies $(1)$, so let $\calh$ be as in $(2)$. The above claim shows that there exists a Borel partition $X^*=\dunion_{i\in I}X_i$ of a conull Borel subset into at most countably many subsets such that for every $i\in I$, $\rho(\calh_{|X_i})$ is contained in a join parabolic subgroup $P_i$. The maximality assumption $(2b)$ together with the implication $(1)\Rightarrow (2a)$ implies that $P_i$ is maximal whenever $X_i$ has positive measure. Indeed, if $P_i$ is contained in a join parabolic subgroup $P'_i$, then $\rho^{-1}(P_i)_{|X_i}\subseteq\rho^{-1}(P'_i)_{|X_i}$, so these two subgroupoids would have to be stably equal. As $\rho$ is action-type, this implies that every element of $P'_i$ has a power contained in $P_i$. By \cite[Lemma~6.4]{Min}, it follows that $P_i=P'_i$, which proves the maximality of $P_i$. Using again the  maximality assumption, after passing to a conull subset and a countable partition, we have $\calh_{|X_i}=\rho^{-1}(P_i)_{|X_i}$.   
\end{proof}

A subgroupoid $\calh$ satisfying one of the equivalent conclusions of Lemma~\ref{lemma:recognize-product} will be called a \emph{maximal join subgroupoid} of $\calg$. Lemma~\ref{lemma:recognize-product} ensures that this notion does not depend on the choice of an action-type cocycle from $\calg$ towards a one-ended nonabelian right-angled Artin group. Notice that the  partition that arises in the first assertion of Lemma~\ref{lemma:recognize-product} is not unique (for instance, one can always pass to a further partition), so it is not determined by the pair $(\calh,\rho)$ in any way. However, the map sending any point $x\in Y_i$ to the parabolic subgroup $P_i$ is determined completely, up to changing its value on a null set. We call it the \emph{parabolic map} of $(\calh,\rho)$. We insist that, while being a maximal join parabolic subgroup is a notion that is independent of the action-type cocycle $\rho$, the parabolic map does depend on $\rho$.

\subsection{Recognizing the center of a right-angled Artin group}

\begin{lemma}\label{lem:center}
Let $G$ be a right-angled Artin group, and let $Z$ be the center of $G$. Let $\calg$ be a measured groupoid over a standard finite measure space $X$, coming with a strict action-type cocycle $\rho:\calg\to G$. Let $\calh\subseteq\calg$ be a measured subgroupoid. Then the following statements are equivalent.
\begin{enumerate}
\item There exist a conull Borel subset $X^*\subseteq X$ and a partition $X^*=\dunion_{i\in I}X_i$ into at most countably many Borel subsets such that for every $i\in I$, one has $\calh_{|X_i}=\rho^{-1}(Z)_{|X_i}$.
\item The following properties hold:
\begin{enumerate}
\item the subgroupoid $\calh$ is amenable and stably normalized by $\calg$; 
\item if $\calh'\subseteq\calg$ is another measured subgroupoid of $\calg$ that satisfies property~(a), and if $\calh$ is stably contained in $\calh'$, then $\calh$ is stably equal to $\calh'$. 
\end{enumerate}
\end{enumerate}
\end{lemma}

A \emph{central subgroupoid} of $\calg$ is a subgroupoid $\calh$  satisfying one of the equivalent conclusions of Lemma~\ref{lem:center}. In the context of Lemma~\ref{lem:center}, if a central subgroupoid $\calh$ is not stably trivial, then $Z$ is infinite. The main point of Lemma~\ref{lem:center} is that the notion of central subgroupoid is independent of the choice of an action-type cocycle from $\calg$ towards a right-angled Artin group.

\begin{proof}
The lemma is clear when $G$ is abelian, so we will assume otherwise. In particular $\calg$ is everywhere nonamenable. As usual, we fix an identification between $G$ and $G_\Gamma$; parabolic subgroups are understood with respect to this identification.

We first observe that $(1)$ implies $(2a)$. Indeed, if $\calh$ is a subgroupoid as in $(1)$, then amenability of $Z$ ensures that $\calh$ is amenable (using that $\rho$ has trivial kernel), and the fact that $Z$ is normal in $G$ ensures that $\calh$ is stably normalized by $\calg$. 

We now claim that if $\calh$ satisfies $(2a)$, then there exists a partition $X^*=\dunion_{i\in I}X_i$ of a conull Borel subset $X^*\subseteq X$ into at most countably many Borel subsets such that for every $i\in I$, one has $\calh_{|X_i}\subseteq\rho^{-1}(Z)_{|X_i}$. Together with the maximality assertion~$(2b)$ and the fact that a subgroupoid as in $(1)$ satisfies $(2a)$, this will show that $(2)\Rightarrow (1)$. This claim will also prove that every subgroupoid as in $(1)$ is stably maximal with respect to $(2a)$, showing that $(1)\Rightarrow (2)$.

We are thus left with proving the above claim. By Lemma~\ref{lemma:support-of-normal-amenable}, there exist a conull Borel subset $X^*\subseteq X$, a partition $X^*=\dunion_{i\in I}X_i$ into at most countably many Borel subsets, and for every $i\in I$, a parabolic subgroup $P_i\subseteq G$ (with respect to the chosen standard generating set), such that 
\begin{enumerate}
\item $\calh_{|X_i}\subseteq\rho^{-1}(P_i)_{|X_i}$,
\item $\calg_{|X_i}\subseteq\rho^{-1}(P_i\times P_i^{\perp})_{|X_i}$, 
\item $\rho^{-1}(P_i)_{|X_i}$ is amenable. 
\end{enumerate}
As $\rho$ is action-type, the second point implies that every element of $G$ has a power contained in $P_i\times P_i^{\perp}$, which in turn implies that $G=P_i\times P_i^{\perp}$ by \cite[Lemma~6.4]{Min}. As $\rho$ is action-type and $\rho^{-1}(P_i)_{|X_i}$ is amenable, the parabolic subgroup $P_i$ does not contain any nonabelian free subgroup, so it is abelian. These two facts together imply that $P_i\subseteq Z$, and the first point above completes our proof.  
\end{proof}

\begin{cor}\label{cor:centerless}
Let $G_1,G_2$ be two right-angled Artin groups. Assume that there exists a measured groupoid $\calg$ which admits two action-type cocycles $\rho_1:\calg\to G_1$ and $\rho_2:\calg\to G_2$.

If $G_1$ has trivial center, then $G_2$ has trivial center. 
\end{cor}

\begin{proof}
 We prove the contrapositive statement, so assume that the center $Z_2$ of $G_2$ is nontrivial. Then $\calz=\rho_2^{-1}(Z_2)$ is a subgroupoid of $\calg$ of infinite type which satisfies assertion~2 from Lemma~\ref{lem:center} (by using the implication $(1)\Rightarrow (2)$ of that lemma, applied to the cocycle $\rho_2$). Using now the implication $(2)\Rightarrow (1)$ from Lemma~\ref{lem:center}, applied to the cocycle $\rho_1$, we deduce that there exists a Borel subset $U\subseteq Y$ of positive measure such that $\rho_1(\calz_{|U})$ is contained in the center $Z_1$ of $G_1$. As $\rho_1$ has trivial kernel and $\calz$ is of infinite type, this implies that $Z_1$ is nontrivial.
\end{proof}

\subsection{Recognizing commuting centers}

\begin{lemma}\label{lemma:recognize-adjacency}
Let $G$ be a one-ended nonabelian right-angled Artin group. Let $\calg$ be a measured groupoid over a standard probability space $X$, and let $\rho:\calg\to G$ be a strict action-type cocycle. Let $\calh,\calh'$ be two maximal join parabolic subgroupoids of $\calg$. Let $X^*\subseteq X$ be a conull Borel subset, and $X^*=\dunion_{i\in I}X_i$ be a partition into at most countably many Borel subsets, such that for every $i\in I$, there exist parabolic subgroups $P_i,P'_i$ of $G$ such that $\calh_{|X_i}=\rho^{-1}(P_i)_{|X_i}$ and $\calh'_{|X_i}=\rho^{-1}(P'_i)_{|X_i}$. Then for every $i\in I$ such that $X_i$ has positive measure, the following assertions are equivalent. 
\begin{enumerate}
\item The centers of $P_i$ and $P'_i$ commute.
\item Given any central subgroupoids $\calz_i\subseteq\calh_{|X_i}$ and $\calz'_i\subseteq\calh'_{|X_i}$ and any Borel subset $U\subseteq X_i$ of positive measure, there exists a Borel subset $V\subseteq U$ of positive measure such that $\langle (\calz_i)_{|V},(\calz'_i)_{|V}\rangle$ is amenable.
\end{enumerate}
\end{lemma}

In the sequel, when two maximal join parabolic subgroupoids of $\calg$ satisfy one of the equivalent conditions of Lemma~\ref{lemma:recognize-adjacency} for every $i\in I$, we say that they are \emph{center-commuting} (notice that this notion does not depend of the choice of a partition as in the statement).

\begin{proof} 
Let $i\in I$ be such that $X_i$ has positive measure, and let $C_i,C'_i$ be the respective centers of $P_i,P'_i$. Let $\hat\calz_i=\rho^{-1}(C_i)_{|X_i}$ and $\hat\calz'_i=\rho^{-1}(C'_i)_{|X_i}$. Notice that  $\calh_{|X_i}$ and $\calh'_{|X_i}$ admit strict action-type cocycles towards $P_i,$ and $P'_i$, respectively. Therefore, Lemma~\ref{lem:center} ensures that  $\hat\calz_i$ and $\hat\calz'_i$ are central subgroupoids of $\calh_{|X_i}$ and $\calh'_{|X_i}$, respectively, and conversely, every central subgroupoid of $\calh_{|X_i}$ or $\calh'_{|X_i}$ is stably equal to $\hat\calz_i$ or $\hat\calz'_i$, respectively. 

Assuming that $(1)$ holds, the group $\langle C_i,C'_i\rangle$ is abelian. Let $\calz_i$ and $\calz'_i$ be central subgroupoids of $\calh_{|X_i}$ and $\calh'_{|X_i}$, respectively. Let $U\subseteq X_i$ be a Borel subset of positive measure, and let $V\subseteq U$ be a Borel subset of positive measure such that $(\calz_i)_{|V}=(\hat\calz_i)_{|V}$ and $(\calz'_i)_{|V}=(\hat\calz'_i)_{|V}$. Then $\langle(\calz_i)_{|V},(\calz'_i)_{|V}\rangle\subseteq\rho^{-1}(\langle C_i,C'_i\rangle)_{|V}$ is amenable (as $\rho$ has trivial kernel). It follows that $(2)$ holds.

Assuming that $(1)$ fails, there exist infinite cyclic subgroups $A_i\subseteq C_i$ and $A'_i\subseteq C'_i$ that together generate a rank $2$ free group: this follows for instance from \cite[Theorem~44]{KK}. Let $V\subseteq X_i$ be any Borel subset of positive measure. It follows from \cite[Lemma~3.20]{Kid} that $\langle \rho^{-1}(A_i)_{|V},\rho^{-1}(A'_i)_{|V}\rangle$ is nonamenable. Therefore $\langle (\hat\calz_i)_{|V},(\hat\calz'_i)_{|V}\rangle$ is nonamenable, showing that $(2)$ fails. 
\end{proof}

\section{Strong rigidity}

In this section, we prove Theorem~\ref{theointro:strong-rigidity}. As explained in the introduction, the $W^*$-rigidity statement follows from the orbit equivalence rigidity statement via \cite[Theorem~1.2 and Remark~1.3]{PV2}, see the argument in the proof of \cite[Corollary~3.20]{HH2} for details. We will therefore focus on the orbit equivalence rigidity statement. Our proof distinguishes two cases, regarding whether or not $G$ contains a maximal join parabolic subgroup with trivial center. Theorem~\ref{theointro:strong-rigidity} is the combination of Propositions~\ref{theo:join-case} and~\ref{theo:coned-case} below.

\subsection{The case where some maximal join parabolic subgroup has trivial center}

We first prove Theorem~\ref{theointro:strong-rigidity} in the case where $G$ contains a maximal join parabolic subgroup with trivial center.

\begin{prop}\label{theo:join-case}
Let $G,H$ be two one-ended right-angled Artin groups. Assume that some maximal join parabolic subgroup of $G$ has trivial center.

Let $G\actson X$ and $H\actson Y$ be two free, irreducible, measure-preserving actions on standard probability spaces. If the actions $G\actson X$ and $H\actson Y$ are stably orbit equivalent (through a stable orbit equivalence $f:U\to V$ between positive measure Borel subsets $U\subseteq X$ and $V\subseteq Y$), then $\kappa(f)=1$, any SOE cocycle associated to $f$ is cohomologous to a group isomorphism, and the actions are actually conjugate through a group isomorphism $\alpha:G\to H$. 
\end{prop}

\begin{proof}
Throughout the proof, we will always identify $G,H$ with right-angled Artin groups $G_\Gamma,G_\Lambda$, in such a way that through these identifications, all standard generators act ergodically on $X,Y$. All parabolic subgroups will be understood with respect to these identifications. 

The groupoid $\calg=(G\ltimes X)_{|U}$ is naturally isomorphic (via $f$) to $(H\ltimes Y)_{|V}$ (after renormalizing the measures on $U$ and $V$). So $\calg$ comes equipped with two action-type cocycles $\rho_G:\calg\to G$ and $\rho_H:\calg\to H$. 

Let $P\subseteq G$ be a maximal join parabolic subgroup with trivial center, and let $\calp=\rho_G^{-1}(P)$. Then $\calp$ satisfies Assertion~2 from Lemma~\ref{lemma:recognize-product}, as it follows from applying this lemma to the cocycle $\rho_G$. Using now the implication $(2)\Rightarrow (1)$ from Lemma~\ref{lemma:recognize-product}, applied to the cocycle $\rho_H$, we see that $\calp$ also satisfies Assertion~1 with respect to $\rho_H$. In particular, there exist a Borel subset $W\subseteq U$ of positive measure and a maximal join parabolic subgroup $Q\subseteq H$ such that $\calp_{|W}=\rho_H^{-1}(Q)_{|f(W)}$. In particular $\calp_{|W}$ comes equipped with two action-type cocycles towards $P$ and $Q$. As $P$ has trivial center, Corollary~\ref{cor:centerless} ensures that $Q$ also has trivial center.

Let $c:G\times X\to H$ be an SOE cocycle associated to $f_{|W}$. Notice that $c$ is also an SOE cocycle associated to $f$, and any two such cocycles are cohomologous. So if we prove that $c$ is cohomologous to a group isomorphism, then the same is true of any SOE cocycle associated to $f$. Notice also that $\kappa(f_{|W})=\kappa(f)$.

The actions $P\actson X$ and $Q\actson Y$ are ergodic (by our irreducibility assumption), and the above ensures that for almost every $x\in W$, one has $f((P\cdot x)\cap W)=(Q\cdot f(x))\cap f(W)$. So $c$ is cohomologous to a cocycle $c'$ such that $c'_{|P\times X}$ is an SOE cocycle associated to $f_{|W}$ for the stable orbit equivalence between $P\actson X$ and $Q\actson Y$ (in particular  $c'(P\times X^*)\subseteq Q$ for some conull Borel subset $X^*\subseteq X$). The groups $P$ and $Q$ are join parabolic subgroups with trivial center, so they split as direct products $P=P_{1}\times\dots\times P_{k}$ and $Q=Q_1\times\dots\times Q_\ell$ of at least two nonabelian parabolic subgroups that do not admit any nontrivial product decomposition. All subgroups $P_i$ and $Q_j$ belong to Monod and Shalom's class $\Creg$, so Theorem~\ref{theo:monod-shalom-product} ensures that $\kappa(f_{|W})=1$ (so $\kappa(f)=1$) and that $c'_{|P\times X}$ is cohomologous to a group isomorphism $\alpha:P\to Q$.  As $P$ contains a conjugate of a standard generator of $G$, Lemma~\ref{lemma:exploting-commuting-gensets} ensures that $c$ is cohomologous to a group isomorphism $\alpha:G\to H$, and the actions $G\actson X$ and $H\actson Y$ are conjugate through $\alpha$. As already mentioned in the previous paragraph, this is enough to conclude.
\end{proof}

\subsection{The case where every maximal join parabolic subgroup has a nontrivial center}

We now prove Theorem~\ref{theointro:strong-rigidity} when every maximal join parabolic subgroup of $G$ has a nontrivial center.

\begin{prop}\label{theo:coned-case}
Let $G,H$ be two one-ended right-angled Artin groups  with trivial center. Assume that every maximal join parabolic subgroup of $G$ has a nontrivial center.

Let $G\actson (X,\mu)$ and $H\actson (Y,\nu)$ be two free, irreducible, measure-preserving actions on standard probability spaces. If the actions $G\actson X$ and $H\actson Y$ are stably orbit equivalent (through a stable orbit equivalence $f:U\to V$ between positive measure Borel subsets $U\subseteq X$ and $V\subseteq Y$), then $\kappa(f)=1$, every SOE cocycle associated to $f$ is cohomologous to a group isomorphism, and the actions are actually conjugate through a group isomorphism $\alpha:G\to H$. 
\end{prop}

\begin{proof} 
As in our previous proof, we will always identify $G,H$ with right-angled Artin groups $G_\Gamma,G_\Lambda$, in such a way that through these identifications, all standard generators act ergodically on $X,Y$. All parabolic subgroups will be understood with respect to these identifications. 

Up to exchanging the roles of $G$ and $H$, we will assume without loss of generality that $\kappa(f)\ge 1$. Let $\calg=(G\ltimes X)_{|U}$, which is naturally isomorphic (through $f$) to $(H\ltimes Y)_{|V}$ (after renormalizing the measures on $U$ and $V$). Then $\calg$ comes equipped with two action-type cocycles $\rho_G:\calg\to G$ and $\rho_H:\calg\to H$.

Lemma~\ref{lemma:combinatorics} ensures that there exist two distinct maximal join parabolic subgroups $P_1,P_2\subseteq G$, whose centers $A_1,A_2$ are infinite and commute, and by Lemma~\ref{lemma:commuting-centers} we have $A_1\cap A_2=\{1\}$. 

For every $i\in\{1,2\}$, let $\calp_i=\rho_G^{-1}(P_i)$ and $\cala_i=\rho_G^{-1}(A_i)$. Then $\calp_1$ and $\calp_2$ are two maximal join subgroupoids of $\calg$ which contain a central subgroupoid of infinite type, and are center-commuting. Using Lemmas~\ref{lemma:recognize-product} and~\ref{lemma:recognize-adjacency}, we can therefore find two distinct maximal join parabolic subgroups $Q_1,Q_2\subseteq H$ with commuting infinite centers $B_1,B_2$, and a Borel subset $W\subseteq U$ of positive measure, such that for every $i\in\{1,2\}$, one has $(\calp_i)_{|W}=\rho_H^{-1}(Q_i)_{|f(W)}$. In addition, Lemma~\ref{lem:center} implies that up to replacing $W$ by a positive measure Borel subset, we can assume that $(\cala_i)_{|W}=\rho_H^{-1}(B_i)_{|f(W)}$ for every $i\in\{1,2\}$.

Let $c:G\times X\to H$ be an SOE cocycle associated to  $f_{|W}$.  Up to replacing $X$ by a conull invariant Borel subset, we can (and will) assume that $c$ is chosen so that whenever $(g,x)\in G\times X$ satisfies $x,gx\in W$, then $c(g,x)$ is the unique element $h\in H$ so that $f(gx)=hf(x)$. We will prove that $c$ is cohomologous to a cocycle $c'$ for which there exists a standard generator $s\in G$ such that $c'(s,\cdot)$ is almost everywhere constant. This will be enough to conclude the proof of our proposition in view of Lemma~\ref{lemma:exploting-commuting-gensets} (after observing that $c$ is also an SOE cocycle associated to $f$, and $\kappa(f_{|W})=\kappa(f)$). 

For $i\in\{1,2\}$, we have $(\calp_i)_{|W}=\rho_H^{-1}(Q_i)_{|f(W)}$, and $P_i,Q_i$ act ergodically on $X,Y$ by assumption. So $c$ is cohomologous to a cocycle $c_i$ such that $(c_i)_{|P_i\times X}$ is an SOE cocycle associated to $f_{|W}$ for the stable orbit equivalence between $P_i\actson X$ and $Q_i\actson Y$, and such that $c$ and $c_i$ coincide on all pairs $(g,x)$ with $x,gx\in W$.

By Lemma~\ref{lem:MS} (applied to the ambient groups $P_i,Q_i$ and to the normal subgroups $A_i,B_i$), for every $i\in\{1,2\}$, up to replacing $c_i$ by a cohomologous cocycle and replacing $X$ by a conull $G$-invariant and $H$-invariant Borel subset, we can assume that the following hold: 
\begin{enumerate}
	\item there is a group isomorphism $\bar\alpha_i:P_i/A_i\to Q_i/B_i$ satisfying that for every $g\in P_i$ and every $x\in X$, one has $c_i(g,x)\in\bar\alpha_i(gA_i)$ -- in particular $c_i(A_i\times X)\subseteq B_i$;
	\item $c_i$ coincides with $c$ on all pairs $(g,x)$ with $x,gx\in W$.
\end{enumerate}

Let $r_i:Q_i=B_i\times B_i^{\perp}\to B_i^{\perp}$ be the retraction, and let $c'_i:A_i^{\perp}\times X\to B_i^{\perp}$ be the cocycle defined as $c'_i=(r_i\circ c_i)_{|A_i^{\perp}\times X}$. There are isomorphisms $P_i/A_i\to A_i^{\perp}$ and $Q_i/B_i\to B_i^{\perp}$ (coming from choosing the unique lift). Through these identifications $\bar\alpha_i$ yields an isomorphism $\alpha_i:A_i^{\perp}\to B_i^{\perp}$ such that for every $g\in A_i^{\perp}$ and every $x\in X$, one has $c'_i(g,x)=\alpha_i(g)$.

Recall from Lemma~\ref{lemma:commuting-centers} that for every $i\in\{1,2\}$, we have $A_{3-i}\subseteq A_i^{\perp}$ and $B_{3-i}\subseteq B_i^{\perp}$. We now prove that for every $i\in\{1,2\}$, the isomorphism  $\alpha_i$ restricts to an isomorphism between $A_{3-i}$ and $B_{3-i}$. By symmetry, it suffices prove it for $i=2$. By Poincaré recurrence, for every $g\in A_1$, there exist an integer $n>0$ and $x\in W$ such that $g^nx\in W$. Then $c_2(g^n,x)=c_1(g^n,x)$ (they are both equal to $c(g^n,x)$), and these belong to $B_1$, which is contained in $B_2^{\perp}$. In particular $c_2(g^n,x)=c'_2(g^n,x)$, which in turn equals $\alpha_2(g)^n$ by the above. So $\alpha_2(g)^n\in B_1$, and therefore $\alpha_2(g)\in B_1$ by \cite[Lemma~6.4]{Min}. So $\alpha_2(A_1)\subseteq B_1$. We now show that actually $\alpha_2(A_1)=B_1$. Take $h\in B_1$. Then there exist an integer $m>0$ and $y\in f(W)$ such that $h^my\in f(W)$. As the actions $A_1\actson X$ and $B_1\actson Y$ induce (via $f$) the same orbit equivalence relation on $W$, there exists $g\in A_1$ such that $c_1(g,y)=c_2(g,y)=h^m$. As $h^m\in B_1\subseteq B_2^{\perp}$, we have $c_2(g,y)=c'_2(g,y)$, so $\alpha_2(g)=h^m$. As $\alpha_2:A_2^\perp\to B_2^\perp$ is an isomorphism, there also exists $g_0\in A_2^\perp$ such that $\alpha_2(g_0)=h$,  hence $g^m_0=g\in A_1$. It follows that $g_0\in A_1$ by \cite[Lemma~6.4]{Min}. This proves that $\alpha_2$ restricts to an isomorphism between $A_1$ and $B_1$, as desired. 

For every $i\in\{1,2\}$, we can therefore extend $\alpha_i$ on $A_i$ by defining $(\alpha_i)_{|A_i}=(\alpha_{3-i})_{|A_i}$ (in particular $\alpha_1$ and $\alpha_2$ coincide on $\langle A_1,A_2\rangle$). This yields an isomorphism $\alpha_i:P_i\to Q_i$ such that for every $g\in A_i\cup A_i^{\perp}$ and every $x\in X$, one has $c_i(g,x)\in\alpha_i(g)B_i$. Now, using the cocycle relation and the fact that every element of $P_i$ is a product of the form $hk$ with $h\in A_i$ and $k\in A_i^{\perp}$, we see that $c_i(g,x)\in\alpha_i(g)B_i$ for every $g\in P_i$ and almost every $x\in X$.

Altogether, these show that, for every $i\in\{1,2\}$, we have a measurable map $\varphi_i:X\to H$ and a measurable map $\kappa_i:P_i\times X\to B_i$ such that for every $g\in P_i$ and almost every $x\in X$, one has $c(g,x)=\varphi_i(gx)\alpha_i(g)\kappa_i(g,x)\varphi_i(x)^{-1}$. Up to replacing $X$ by a conull $G$-invariant Borel subset, we will assume that these relations hold for every $g\in G$ and every $x\in X$. Let $X=\dunion_{j\in J}X_j$ be a partition into at most countably many Borel subsets such that for every $j\in J$, the maps $\varphi_1,\varphi_2$ have constant values $\gamma_{1,j},\gamma_{2,j}$ when restricted to $X_j$. 

Let $g\in A_1$ be a nontrivial element. By Poincaré recurrence, for every $j\in J$ such that $X_j$ has positive measure, there exist an integer $k_j>0$ and $x\in X_j$ such that $g^{k_j}x\in X_j$. By observing that $A_1\subseteq P_1\cap P_2$, we can then write $$c(g^{k_j},x)=\gamma_{1,j} \alpha_1(g)^{k_j}\kappa_1(g^{k_j},x)\gamma_{1,j}^{-1}=\gamma_{2,j} \alpha_2(g)^{k_j}\kappa_2(g^{k_j},x)\gamma_{2,j}^{-1}$$ where  $\alpha_1(g)^{k_j}=\alpha_2(g)^{k_j}$ belongs to $B_1\subseteq B_2^{\perp}$, and $\kappa_1(g^{k_j},x)\in B_1$ and $\kappa_2(g^{k_j},x)\in B_2$. Let $r_2:H\to B_2$ be the retraction as in Section~\ref{sec:raag}. As $B_1\subseteq B_2^{\perp}$, we have $r_2(B_1)=\{1\}$. By applying $r_2$ to the above equation, we deduce that $\kappa_2(g^{k_j},x)$ is trivial. Now applying $r_1:H\to B_1$ to the above equation, and using the fact that $\alpha_1(g)=\alpha_2(g)$ and $B_1$ is abelian, we deduce that $\kappa_1(g^{k_j},x)$ is trivial. Therefore $\gamma_{1,j}^{-1}\gamma_{2,j}$ commutes with $\alpha_1(g)^{k_j}$.  

We claim that the centralizer $Z$ of $\alpha_1(g)^{k_j}$ is a join parabolic subgroup. Indeed, let $C_1\subseteq B_1$ be the smallest parabolic subgroup that contains $\alpha_1(g)^{k_j}$. As $B_1$ (whence $C_1$) is abelian, we have $Z=C_1\times C_1^{\perp}$. In addition $C_1$ is nontrivial (because $g$ is nontrivial), and $C_1^\perp$ is also nontrivial because it contains $B_1^{\perp}$. This proves our claim. As $Z$ contains $Q_1=B_1\times B_1^{\perp}$, it follows that $Z=Q_1$ by the maximality of $Q_1$ as a join parabolic subgroup. 

This proves that for almost every $x\in X$, one has $\varphi_2(x)=\varphi_1(x)\eta_1(x)\mu_1(x)$, where $\eta_1(x)\in B_1$ and $\mu_1(x)\in B_1^{\perp}$. Now, for every $g\in A_1$ and every $x\in X$, one has 
\begin{displaymath}
\begin{array}{rl}
c(g,x)&=\varphi_1(gx)\alpha_1(g)\kappa_1(g,x)\varphi_1(x)^{-1}\\
& =\varphi_1(gx)\eta_1(gx)\mu_1(gx)\alpha_1(g)\kappa_2(g,x)\mu_1(x)^{-1}\eta_1(x)^{-1}\varphi_1(x)^{-1},
\end{array}
\end{displaymath}
\noindent and therefore $$\kappa_1(g,x)=\eta_1(gx)\mu_1(gx)\kappa_2(g,x)\mu_1(x)^{-1}\eta_1(x)^{-1}.$$ Retracting to $B_1$ yields $$\kappa_1(g,x)=\eta_1(gx)\eta_1(x)^{-1},$$ and therefore $$c(g,x)=\varphi_1(gx)\eta_1(gx)\alpha_1(g)\eta_1(x)^{-1}\varphi_1(x)^{-1}.$$ This proves that there exists a measurable map $\psi:X\to H$ and a homomorphism $\alpha_1:A_1\to H$ such that for every $g\in A_1$ and every $x\in X$, one has $c(g,x)=\psi(gx)\alpha_1(g)\psi(x)^{-1}$, which concludes our proof.
\end{proof}

\section{Superrigidity}

In this section, we derive Theorem~\ref{theointro:superrigidity} from Theorem~\ref{theointro:strong-rigidity}, using general techniques developed in prior works of Furman \cite{Fur-oe}, Monod and Shalom \cite{MS} and Kida \cite{Kid-oe}.

Let $G$ be a countable group, and let $\calf$ be a collection of subgroups of $G$. One says that a free, ergodic, measure-preserving action of $G$ on a standard probability space $X$ is \emph{$\calf$-ergodic} if every subgroup in $\calf$ acts ergodically on $X$. In our setting, an irreducible action of a right-angled Artin group is an action which is $\calf$-ergodic with respect to the collection of all cyclic subgroups associated to a standard generating set $\calf$.

One says that $(G,\calf)$ is \emph{strongly cocycle-rigid} if given any two stably orbit equivalent $\calf$-ergodic free, ergodic, measure-preserving actions $G\actson X$ and $G\actson Y$ on standard probability spaces, any SOE cocycle $c:G\times X\to G$ is cohomologous to a group isomorphism $\alpha:G\to G$. Notice that Propositions~\ref{theo:join-case} and~\ref{theo:coned-case} imply that if $G$ is a one-ended right-angled Artin group with trivial center, and if $\calf$ is the set of all cyclic subgroups associated to standard generators of $G$ (under an isomorphism between $G$ and some $G_\Gamma$), then $(G,\calf)$ is strongly cocycle-rigid.

Recall that a free, ergodic, measure-preserving action of a countable group $G$ on a standard probability space $(X,\mu)$ is \emph{mildly mixing} if for every Borel subset $A\subseteq X$, and every sequence $(g_n)_{n\in\mathbb{N}}\in G^{\mathbb{N}}$ made of pairwise distinct elements, either $A$ is null or conull, or else $\liminf_{n\to \infty}\mu(g_nA\bigtriangleup A)>0$. This is equivalent to requiring that for every non-singular properly ergodic action of $G$ on a standard probability measure space $Y$, the diagonal $G$-action on $X\times Y$ is ergodic \cite{SW}. Every mildly mixing $G$-action is  $\mathcal F$-ergodic, taking for $\mathcal F$ the collection of all infinite subgroups of $G$.

\begin{theo}
	\label{theo:conj}
Let $G$ be an ICC countable group, and let $\calf$ be a collection of infinite subgroups of $G$. Assume that $(G,\calf)$ is strongly cocycle-rigid. Let $H$ be a countable group. Let $X,Y$ be standard probability spaces, let $G\actson X$ be an $\calf$-ergodic free, ergodic, measure-preserving $G$-action, and let $H\actson Y$ be a free, measure-preserving, mildly mixing $H$-action.

If the actions $G\actson X$ and $H\actson Y$ are stably orbit equivalent, then they are virtually conjugate. 
\end{theo} 

\begin{proof}
By \cite[Theorem~3.3]{Fur-oe}, there exists a standard measure space $\Sigma$ equipped with a measure-preserving action of $G\times H$ such that the $G$-action on $X$ is isomorphic to the $G$-action on $H\backslash\Sigma$, and the $H$-action on $Y$ is isomorphic to the $H$-action on $G\backslash\Sigma$ (the space $\Sigma$ is a \emph{measure equivalence coupling} between $G$ and $H$ in the sense of \cite[0.5.E]{Gro}). Let $\Omega$ be the self measure equivalence coupling of $G$ defined by $\Omega=\Sigma\times_{H}H\times_{H}\check{\Sigma}$ (see \cite[Section~2]{Fur-me} for definitions). By definition $\Omega$ comes equipped with a measure-preserving action of $G\times G$; for notational simplicity, we will let $G_{\ell}=G\times\{1\}$ and $G_r=\{1\}\times G$. As the $H$-action on $Y$ is mildly mixing, \cite[Lemma~6.5]{MS} ensures that the actions of $G_\ell$ on $G_r\backslash\Omega$, and of $G_r$ on $G_\ell\backslash\Omega$, are ergodic and $\calf$-ergodic. In addition, the essential freeness of the $G$-action on $H\backslash\Sigma$ ensures that the actions of $G_\ell$ on $G_r\backslash\Omega$ and of $G_r$ on $G_\ell\backslash\Omega$ are essentially free.

We now claim that there exist a Borel map $\Phi:\Omega\to G$ and an automorphism $\rho:G\to G$ such that $\Phi$ is \emph{$\rho$-twisted equivariant}, i.e.\ $\Phi$ is $(G\times G)$-equivariant when $G$ is equipped with the action of $G\times G$ given by $(g_1,g_2)\cdot g=\rho(g_1)gg_2^{-1}$. Let $Z\subseteq\Omega$ be a fundamental domain for the action of $G_r$.  By identifying $Z$ with $G_r\backslash\Omega$, we get an essentially free, ergodic, $\calf$-ergodic, measure-preserving action of $G_\ell$ on $Z$. It follows from \cite[Lemma~3.2]{Fur-oe} that there exist an SOE cocycle $c:G_\ell\times Z\to G_r$ and a $(G_\ell\times G_r)$-equivariant Borel isomorphism $\Omega\to Z\times G_r$, where the action of $G_\ell\times G_r$ on $Z\times G_r$ is given by $(g_1,g_2)\cdot (z,g)=(g_1z,c(g_1,z)gg_2^{-1})$. As the actions of $G$ on $G_\ell\backslash\Omega$ and $G_r\backslash\Omega$ are  $\calf$-ergodic, and $(G,\calf)$ is strongly cocycle-rigid, the cocycle $c$ is cohomologous to a group isomorphism, i.e.\ there exist a group isomorphism $\rho:G_\ell\to G_r$ and a measurable map $\varphi:Z\to G_r$ such that for every $g_1\in G_\ell$ and almost every $z\in Z$, one has  $c(g_1,z)=\varphi(g_1z)\rho(g_1)\varphi(z)^{-1}$. We define  $\Phi(z,g)=\varphi(z)^{-1}g$. Then the equivariance is verified as follows: $\Phi((g_1,g_2)\cdot (z,g))=\Phi(g_1z,c(g_1,z)gg_2^{-1})=\rho(g_1)\varphi(z)^{-1}gg_2^{-1}=(g_1,g_2)\cdot \Phi(g,z)$. This proves our claim.

We can thus apply \cite[Theorem~6.1]{Kid} (or the reasoning on \cite[pp.865-867]{MS}) to get a homomorphism $\alpha:H\to G$ with finite kernel and finite-index image, and an almost $(G\times H)$-equivariant Borel map $\Sigma\to G$, where the action of $G\times H$ on $G$ is via $(g,h)\cdot g'=gg'\alpha(h)^{-1}$. The conclusion then follows  from \cite[Lemma~4.18]{Fur-survey} (alternatively, see the argument from the proof of \cite[Theorem~1.1]{Kid-oe}).
\end{proof}

We can now complete the proof of Theorem~2 from the introduction.

\begin{proof}[Proof of Theorem~\ref{theointro:superrigidity}]
 By definition of irreducibility of the $G$-action on $X$, there exists an isomorphism between a right-angled Artin group $G_\Gamma$ and $G$ such that, letting $\calf$ be the set of all cyclic subgroups of $G$ generated by the images (under this identification) of the standard generators of $G_\Gamma$, the $G$-action on $X$ is $\calf$-ergodic.   Propositions~\ref{theo:join-case} and~\ref{theo:coned-case} ensure that $(G,\calf)$ is strongly cocycle-rigid. In addition $G$ is ICC (Lemma~\ref{lemma:icc}). So Theorem~\ref{theo:conj} applies and yields the orbit equivalence superrigidity statement. The $W^*$-superrigidity statement follows because $L^\infty(X)\rtimes G$ contains a unique virtual Cartan subalgebra up to unitary conjugacy (by \cite[Theorem~1.2 and Remark~1.3]{PV2}, see also the proof of \cite[Corollary~3.20]{HH2}). 
\end{proof} 




\section{$W^*$-rigidity results for Bernoulli actions}
%

In this final section, we establish the $W^*$-rigidity theorem given in Theorem~\ref{theointro:Bernoulli} of the introduction of the paper. We first recall a cocycle superrigidity theorem due to Popa.

\begin{theo}[{Popa \cite[Theorem~1.1]{Pop}}]\label{theo:cocycle-popa}
Let $G$ be a countable group that admits a chain of infinite subgroups $G_0\subseteq G_1\subseteq \dots\subseteq G_n=G$ such that 
\begin{enumerate}
\item for every $k\in\{1,\dots,n\}$, the set of all $g\in G_k$ with $|gG_{k-1}g^{-1}\cap G_{k-1}|=\infty$ generates $G_k$, and 
\item the centralizer of $G_0$ in $G$ is nonamenable.
\end{enumerate}
Let $G\actson X$ be a nontrivial Bernoulli action, and let $H$ be any countable group. Then any cocycle $c:G\times X\to H$ is cohomologous to a group homomorphism. 
\end{theo}   

Notice that in the terminology from \cite{Pop}, our assumption precisely says that $G_0$ is \emph{wq-normal} in $G$. 

In particular, if a free, ergodic, probability measure-preserving action $H\actson Y$ on a standard probability space $Y$ is orbit equivalent to $G\actson X$ and if $G$ has no finite nontrivial normal subgroup, then the two actions are conjugate. Likewise, if $G\actson X$ and $H\actson Y$ are stably orbit equivalent, then they are in fact virtually conjugate. 

Theorem~\ref{theo:cocycle-popa} applies to every countable group $G$ which has a finite chain-commuting generating set $S$ whose elements all have infinite order, and such that $S$ contains an element $s_0$ with nonamenable centralizer in $G$. Indeed, this is proved by writing $S=\{s_0,\dots,s_n\}$ in such a way that $s_i$ and $s_{i+1}$ commute for all $i\in\{0,\dots,n-1\}$, letting $G_0=\langle s_0\rangle$, and letting $G_i=\langle G_{i-1},s_i\rangle$ for every $i\in\{1,\dots,n\}$. In particular, it applies to all one-ended nonabelian right-angled Artin groups. For these, $W^*$-superrigidity of all nontrivial Bernoulli actions then follows from the uniqueness of the virtual Cartan subalgebra of $L^\infty(X)\rtimes G$ up to unitary conjugacy. We now present another approach to $W^*$-superrigidity which bypasses Cartan-rigidity (thereby providing new examples), which is a slight variation on \cite[Theorem~10.1]{IPV} (whose notation we now follow for convenience). A free, ergodic, measure-preserving action $G\actson X$ of a countable group on a standard probability space is \emph{$W^*$-superrigid} if for every other free, ergodic, measure-preserving action $H\actson Y$, if $L^\infty(X)\rtimes G\approx L^\infty(Y)\rtimes H$, then the actions $G\actson X$ and $H\actson Y$ are conjugate. 

\begin{theo}[\cite{IPV}]\label{theo:ipv} 
Let $\Gamma$ be an ICC countable group that admits a chain of infinite subgroups $\Gamma_0\subseteq\Gamma_1\subseteq\cdots\subseteq\Gamma_n=\Gamma$ such that for every $k\in\{1,\dots,n\}$, the set of all $g\in\Gamma_k$ with $|g\Gamma_{k-1}g^{-1}\cap\Gamma_{k-1}|=\infty$ generates $\Gamma_k$, and the centralizer of $\Gamma_0$ in $\Gamma$ is nonamenable. 

Then every nontrivial Bernoulli action of $\Gamma$ is $W^*$-superrigid.
\end{theo}

\begin{proof}
We assume the notation from the proof of \cite[Theorem~10.1]{IPV}. 
Let $\Gamma\curvearrowright X$ be a nontrivial Bernoulli action,
denote $M=L^{\infty}(X)\rtimes \Gamma$ and endow $M$ with its canonical trace $\tau$ and the associated $2$-norm, $\|\cdot\|_2$.
 Let $\pi:L^{\infty}(Y)\rtimes \Lambda\rightarrow M$ be a $*$-isomorphism, for some free, ergodic, measure-preserving action $\Lambda\curvearrowright Y$ of a countable group on a standard probability space. Identify $M=L^{\infty}(Y)\rtimes \Lambda$ via $\pi$. Let $\Delta:M\rightarrow M\overline{\otimes}M$ be the unital $*$-homomorphism given by $\Delta(bv_s)=bv_s\otimes v_s$, for every $b\in L^\infty(Y)$ and $s\in \Lambda$ \cite{PV}. Here, $(v_s)_{s\in \Lambda}$ denote the canonical unitaries used to define $L^\infty(Y)\rtimes \Lambda$. 

Below, we also denote by $(u_g)_{g\in \Gamma}$ the canonical unitaries used to define $L^\infty(X)\rtimes\Gamma$ and by $(\sigma_g)_{g\in\Gamma}$ the Bernoulli action of $\Gamma$ on $L^\infty(X)$.
Additionally, we use the notation $A\prec B$, for von Neumann subalgebras $A,B\subset M\overline{\otimes}M$, to mean that {\it a corner of $A$ embeds into $B$ inside $M\overline{\otimes}M$}  in the sense of Popa \cite{Pop-T} (see also \cite[Definition~2.1]{IPV}).

Since the action $\Gamma\curvearrowright X$ satisfies Popa's cocycle superrigidity theorem (Theorem~\ref{theo:cocycle-popa} above), the proof of \cite[Theorem~10.1]{IPV} shows that in order to derive the conclusion we only need to justify:
\\
\\
\textit{Step 1.} There exists a unitary $v\in M\overline{\otimes}M$ such that $v\Delta(\text{L}\Gamma)v^*\subset \text{L}\Gamma\overline{\otimes}\text{L}\Gamma$.
\\

Denote by $P$ the quasi-normalizer of $\Delta(\text{L}\Gamma_0)$ inside $M\overline{\otimes}M$. Since the centralizer $C$ of $\Gamma_0$ in $\Gamma$ is non-amenable, $\text{L}C$ has no amenable direct summand. By \cite[Lemma~10.2(5)]{IPV} we get that $\Delta(\text{L}C)$ is strongly nonamenable relative to $M\otimes 1$ and $1\otimes M$, and that $\Delta(\text{L}C)\nprec M\overline{\otimes}\text{L}^{\infty}(X)$ and $\Delta(\text{L}C)\nprec \text{L}^{\infty}(X)\overline{\otimes}M$. Using that
$\Delta(\text{L}C)\subset P$,  we derive that $P\nprec M\overline{\otimes}\text{L}^{\infty}(X)$ and $P\nprec \text{L}^{\infty}(X)\overline{\otimes}M$.
The rest of Step 1 in the proof of \cite[Theorem~10.1]{IPV} now  applies to show the existence of a unitary $v\in M\overline{\otimes}M$ such that $vPv^*\subset \text{L}\Gamma\overline{\otimes}\text{L}\Gamma$. 
Thus, we have $v\Delta(\text{L}\Gamma_0)v^*\subset \text{L}\Gamma\overline{\otimes}\text{L}\Gamma$.

We prove by induction that $P_k:=v\Delta(\text{L}\Gamma_k)v^*$ is contained in $\text{L}\Gamma\overline{\otimes}\text{L}\Gamma$ for all $k\in\{0,\dots,n\}$. Assume that $P_k\subset \text{L}\Gamma\overline{\otimes}\text{L}\Gamma$ for some $k$ with $0\leq k\leq n-1$. Let $g\in \Gamma_{k+1}$ such that $\Sigma:=g\Gamma_k g^{-1}\cap\Gamma_k$ is infinite and denote $R=v\Delta(\text{L}\Sigma)v^*$. Since $\Sigma\subset \Gamma_k$ we have that $R\subset \text{L}\Gamma\overline{\otimes}\text{L}\Gamma$.  Since $\Sigma$ is infinite, Step 1 in the proof of \cite[Theorem~10.1]{IPV} shows that $R\nprec M\otimes 1$ and $R\nprec 1\otimes M$. We continue with the following claim:

\vskip 0.1in
{\bf Claim.} The inclusion $\text{L}\Gamma\overline{\otimes}\text{L}\Gamma\subset M\overline{\otimes}M$ is weakly mixing through $R$ in the sense of \cite[Definition~6.13]{PV3}.
\vskip 0.1in

{\it Proof of the claim.}
 Since $R\nprec M\otimes 1$ and $R\nprec 1\otimes M$, by  Fact 1 in the proof of \cite[Theorem 4.3]{IPP} there is a sequence $u_n\in\mathcal U(R)$ such that for all $a,b\in M\overline{\otimes}M$, we have $\|E_{M\otimes 1}(au_nb)\|_2\rightarrow 0$ and $\|E_{1\otimes M}(au_nb)\|_2\rightarrow 0$. As $R\subset\text{L}\Gamma\overline{\otimes}\text{L}\Gamma$, we can write $u_n=\sum_{g,h\in\Gamma}c_{g,h}^n(u_g\otimes u_h)$, where the coefficients $c_{g,h}^n\in\mathbb C$ for every $g,h\in\Gamma$ satisfy $\sum_{g,h\in\Gamma}|c_{g,h}^n|^2=\|u_n\|_2^2=1$. If $g\in\Gamma$, then $\text{E}_{1\otimes M}(u_n(u_g^*\otimes 1))=1\otimes (\sum_{h\in\Gamma}c_{g,h}^nu_h)$ and so  $\sum_{h\in\Gamma}|c_{g,h}^n|^2=\|E_{1\otimes M}(u_n(u_g^*\otimes 1))\|_2^2\rightarrow 0$. Similarly, $\sum_{g\in\Gamma}|c_{g,h}^n|^2\rightarrow 0$, for every $h\in\Gamma$.

 We will show that
$$\|\text{E}_{\text{L}\Gamma\overline{\otimes}\text{L}\Gamma}(xu_ny)\|_2\rightarrow 0, \;\;\text{for every} \;\;x\in (M\overline{\otimes}M)\ominus (\text{L}\Gamma\overline{\otimes}\text{L}\Gamma)\;\;\text{and}\;\; y\in M\overline{\otimes} M.$$
which by \cite[Definition~6.13]{PV3} implies the above claim. Let $(\text{L}^\infty(X))_1$ be the operator norm unit ball of $\text{L}^\infty(X)$, i.e., the set of $a\in\text{L}^\infty(X)$ with $\|a\|\leq 1$.
Then the linear span of $\{(u_g\otimes u_h)(a\otimes b)\mid g,h\in\Gamma, a,b\in (\text{L}^{\infty}(X))_1, \tau(a)\tau(b)=0\}$ is $\|\cdot\|_2$-dense in  
$(M\overline{\otimes}M)\ominus (\text{L}\Gamma\overline{\otimes}\text{L}\Gamma)$ and the linear span of $\{(c\otimes d)(u_k\otimes u_l)\mid k,l\in\Gamma, c,d\in (\text{L}^{\infty}(X))_1\}$ is $\|\cdot\|_2$-dense in  $M\overline{\otimes}M$.
Since $\text{E}_{\text{L}\Gamma\overline{\otimes}\text{L}\Gamma}$ is $\text{L}\Gamma\overline{\otimes}\text{L}\Gamma$-bimodular, for every $v,w\in\text{L}(\Gamma)\overline{\otimes}\text{L}(\Gamma)$ and $a,b,c,d\in \text{L}^\infty(X)$, we have $\text{E}_{\text{L}\Gamma\overline{\otimes}\text{L}\Gamma}(v(a\otimes b)u_n(c\otimes d)w)=v\text{E}_{\text{L}\Gamma\overline{\otimes}\text{L}\Gamma}((a\otimes b)u_n(c\otimes d))w$.
By combining the last two facts we see that it suffices to prove the displayed convergence for every $x$ and $y$ of the form $x=a\otimes b,y=c\otimes d$, where $a,b,c,d\in (\text{L}^\infty(X))_1$ are such that $\tau(a)\tau(b)=0$. 

Assume that $\tau(a)=0$ since the case $\tau(b)=0$ is similar.
Let $\varepsilon>0$. For every $n$ we have $\text{E}_{\text{L}\Gamma\overline{\otimes}\text{L}\Gamma}(xu_ny)
=\sum_{g,h\in\Gamma}c_{g,h}^n\tau(a\sigma_g(c))\tau(b\sigma_h(d))(u_g\otimes u_h)$
and thus $$\|\text{E}_{\text{L}\Gamma\overline{\otimes}\text{L}\Gamma}(xu_ny)\|_2^2=\sum_{g,h\in\Gamma}|c_{g,h}^n|^2|\tau(a\sigma_g(c))|^2|\tau(b\sigma_h(d))|^2.$$
Since the Bernoulli action $(\sigma_g)_{g\in\Gamma}$ is mixing and $\tau(a)=0$, we can find a finite set $F\subset \Gamma$ such that $|\tau(a\sigma_g(c))|\leq\varepsilon$, for every $g\in\Gamma\setminus F$. Since $a,b,c,d\in (\text{L}^\infty(X))_1$  we also have that $|\tau(a\sigma_g(c))|\leq 1$ and $|\tau(b\sigma_h(d))|\leq 1$, for every $g,h\in\Gamma$. Altogether, we get $$\|\text{E}_{\text{L}\Gamma\overline{\otimes}\text{L}\Gamma}(xu_ny)\|_2^2\leq\sum_{g\in F, h\in \Gamma}|c_{g,h}^n|^2+\varepsilon^2\cdot\sum_{g\in \Gamma\setminus F,h\in\Gamma}|c_{g,h}^n|^2.$$ Since $\sum_{h\in\Gamma}|c_{g,h}^n|^2\rightarrow 0$, for every $g\in F$, and $\sum_{g\in \Gamma\setminus F,h\in\Gamma}|c_{g,h}^n|^2\leq\sum_{g,h\in\Gamma}|c_{g,h}^n|^2=1$, it follows that $\limsup_n\|\text{E}_{\text{L}\Gamma\overline{\otimes}\text{L}\Gamma}(xu_ny)\|_2^2\leq\varepsilon^2.$ Since this holds for every $\varepsilon>0$, we conclude that $\|\text{E}_{\text{L}\Gamma\overline{\otimes}\text{L}\Gamma}(xu_ny)\|_2\rightarrow 0$, which proves the above claim. \qed
\\

We can now complete the proof of Theorem \ref{theo:ipv}.
 Since $\Sigma g\subset g\Gamma_k$,  we get that $R(v\Delta(u_g)v^*)\subset (v\Delta(u_g)v^*)P_k\subset (v\Delta(u_g)v^*)(\text{L}\Gamma\overline{\otimes}\text{L}\Gamma)$. 
Using the claim, a weak mixing technique due to Popa (see \cite[Proposition~6.14]{PV3})  implies that $v\Delta(u_g)v^*\in \text{L}\Gamma\overline{\otimes}\text{L}\Gamma$, for every $g\in \Gamma_{k+1}$ such that $\Sigma:=g\Gamma_k g^{-1}\cap\Gamma_k$ is infinite. 
Since the set of such $g\in\Gamma_{k+1}$ generates $\Gamma_{k+1}$ we conclude that $P_{k+1}\subset \text{L}\Gamma\overline{\otimes}\text{L}\Gamma$. 

For $k=n$, we get that $v\Delta(\text{L}\Gamma)v^*\subset \text{L}\Gamma\overline{\otimes}\text{L}\Gamma$, which finishes the proof.
\end{proof}

As above, we record the following consequence, which is Theorem~\ref{theointro:Bernoulli} from the introduction.

\begin{cor}\label{cor:bernoulli-chain-commuting}
Let $G$ be an ICC countable group. Assume that $G$ has a finite chain-commuting generating set consisting of infinite-order elements, one of which has nonamenable centralizer in $G$. 

Then every nontrivial Bernoulli action of $G$ is $W^*$-superrigid. \qed
\end{cor}

\begin{cor}
Let $S=S_{g,n}$ be a surface obtained from a closed orientable surface of genus $g$ by removing $n$ points. Assume that $3g+n-5\ge 0$. 

Then every nontrivial Bernoulli action of the mapping class group $\Mod(S_{g,n})$ is $W^*$-superrigid.
\end{cor}

\begin{proof}
The fact that mapping class groups are ICC was proved by Kida in \cite[Theorem~2.9]{Kid}. The fact that they admit chain-commuting generating sets as in Corollary~\ref{cor:bernoulli-chain-commuting} follows for instance from \cite[Corollary~2.11]{LP} when $g\ge 1$, and from \cite[Lemma~23]{Waj} when $g=0$ (note that each element $h$ in the generating set is either a Dehn twist or a braid twist, so its centralizer is non-amenable by considering mapping classes supported on a subsurface which is disjoint from the support of $h$). The conclusion thus follows from Corollary~\ref{cor:bernoulli-chain-commuting}.
\end{proof}

Recall that given a finite labeled simple graph $\Gamma$, where every edge is labeled by an integer at least $2$, the \emph{Artin group} $G_\Gamma$ with defining graph $\Gamma$ is the group defined by the following presentation: it has one generator per vertex of $\Gamma$, with a relation $uvu\dots=vuv\dots$ (with $n$ letters on each side) whenever the vertices $u,v$ are joined by an edge labeled $n$ in $\Gamma$. Whenever $\Lambda\subseteq\Gamma$ is a full subgraph (i.e.\ two vertices of $\Lambda$ are adjacent in $\Lambda$ if and only if they are adjacent in $\Gamma$) with the induced labeling, then the natural homomorphism $G_\Lambda\to G_\Gamma$ induced by the inclusion $\Lambda\hookrightarrow\Gamma$ is injective \cite{vdL}.

\begin{cor}\label{cor:Artin}
Let $G$ be a one-ended ICC Artin group. Then every nontrivial Bernoulli action of $G$ is $W^*$-superrigid.  
\end{cor}

Many classes of Artin groups are known to be ICC. For instance, this is satisfied whenever $G$ is acylindrically hyperbolic and has no nontrivial finite normal subgroup \cite[Theorem~2.35]{DGO}. See e.g.\ \cite{Cal} and the references therein for the current status of known results regarding acylindrical hyperbolicity of Artin groups. The lack of nontrivial finite normal subgroup would follow from the $K(\pi,1)$-conjecture for Artin groups, see \cite{GP,Paris} for surveys of known cases and \cite{PaS} for the most recent developments. Combining these references, we see that all Artin groups with trivial center and connected defining graph that are either of Euclidean type, 2-dimensional, or of type FC, are ICC and satisfy Corollary~\ref{cor:Artin}.

\begin{proof}
Write $G=G_\Gamma$ for some finite labeled simple graph $\Gamma$. Note that $G$ being one-ended implies that $\Gamma$ is connected (otherwise $G$ is a free product).	

If all edges of $G$ are labeled by $2$, then $G$ is a right-angled Artin group. In this case $\Gamma$ is not a clique (otherwise $G$ is free abelian, whence not ICC). Thus the standard generating set is a chain-commuting generating set such that one generator has a non-amenable centralizer, and the conclusion follows from Corollary~\ref{cor:bernoulli-chain-commuting}.

We now assume that $\Gamma$ has at least one edge with label at least $3$. Whenever $e\subseteq \Gamma$ is an edge with label at least $3$, the edge group $G_e$ is non-amenable and has an infinite cyclic center $Z_e$, see \cite[Section~2]{Cri}. The generating set $S$ of $G$ consisting of the generators associated to the vertices of $\Gamma$, together with a generator of $Z_e$ for every edge $e$ with label at least $3$, satisfies the assumption of Corollary~\ref{cor:bernoulli-chain-commuting}, so the conclusion follows.   
\end{proof}


\footnotesize

\bibliographystyle{alpha}
\bibliography{mixing-4}

\begin{flushleft}
Camille Horbez\\
Universit\'e Paris-Saclay, CNRS,  Laboratoire de math\'ematiques d'Orsay, 91405, Orsay, France \\
\emph{e-mail:~}\texttt{camille.horbez@universite-paris-saclay.fr}\\[4mm]
\end{flushleft}

\begin{flushleft}
Jingyin Huang\\
Department of Mathematics\\
The Ohio State University, 100 Math Tower\\
231 W 18th Ave, Columbus, OH 43210, U.S.\\
\emph{e-mail:~}\texttt{huang.929@osu.edu}\\
\end{flushleft}

\begin{flushleft}
Adrian Ioana\\
Department of Mathematics, University of California San Diego, 9500 Gilman Drive, La Jolla, CA 92093, USA\\
\emph{e-mail:~}\texttt{aioana@ucsd.edu}
\end{flushleft}

\end{document}